\documentclass[12pt,reqno]{amsart}

\usepackage{amsmath, amsopn, amssymb, amsthm}
\usepackage{enumerate}
\usepackage{hyperref, verbatim}

\newtheorem{thm}{Theorem}[section]

\newtheorem*{SpSzRL}{The sparse Szemer\'edi regularity lemma}
\newtheorem*{EL}{The embedding lemma}

\newtheorem{lemma}[thm]{Lemma}
\newtheorem{cor}[thm]{Corollary}
\newtheorem{conj}[thm]{Conjecture}
\newtheorem{prop}[thm]{Proposition}

\theoremstyle{definition}

\newtheorem*{WAlg}{The Scythe Algorithm}
\newtheorem*{Const}{Construction}
\newtheorem*{Proc}{Constructing $g_0$ and $f_0^*$}
\newtheorem*{Prop}{Properties}
\newtheorem{defn}[thm]{Definition}

\newtheorem{remark}[thm]{Remark}
\newtheorem*{claim*}{Claim}

\DeclareMathOperator{\ex}{ex}
\DeclareMathOperator{\Aut}{Aut}

\newcommand{\Ex}{\mathbb{E}}
\renewcommand{\Pr}{\mathbb{P}}

\newcommand{\ol}{\overline}

\newcommand{\eps}{\varepsilon}

\newcommand{\A}{\mathcal{A}}
\newcommand{\B}{\mathcal{B}}
\newcommand{\F}{\mathcal{F}}
\newcommand{\G}{\mathcal{G}}
\newcommand{\Gt}{\tilde{G}}
\newcommand{\HH}{\mathcal{H}}

\newcommand{\I}{\mathcal{I}}
\newcommand{\N}{\mathbb{N}}

\newcommand{\p}{\mathbf{p}}
\newcommand{\cP}{\mathcal{P}}
\newcommand{\ZZ}{\mathbb{Z}}
\renewcommand{\S}{\mathcal{S}}

\renewcommand{\le}{\leqslant}
\renewcommand{\ge}{\geqslant}

\title{Independent sets in hypergraphs}

\date{\today}

\author{J\'ozsef Balogh}
\address{Department of Mathematics, University of Illinois, 1409 W. Green Street, Urbana, IL 61801} \email{jobal@math.uiuc.edu}

\author{Robert Morris} 
\address{IMPA, Estrada Dona Castorina 110, Jardim Bot\^anico, Rio de Janeiro, RJ, Brasil} \email{rob@impa.br}

\author{Wojciech Samotij}
\address{School of Mathematical Sciences, Tel Aviv University, Tel Aviv 69978, Israel; and Trinity College, Cambridge CB2 1TQ, UK} \email{ws299@cam.ac.uk}

\thanks{Research supported in part by: (JB) NSF CAREER Grant DMS-0745185, UIUC Campus Research Board Grant 11067, and OTKA Grant K76099; (RM) CNPq bolsa de Produtividade em Pesquisa; (WS) ERC Advanced Grant DMMCA and a Trinity College JRF}

\begin{document}

\begin{abstract}
  Many important theorems and conjectures in combinatorics, such as the theorem of Szemer\'edi on arithmetic progressions and the Erd{\H{o}}s-Stone Theorem in extremal graph theory, can be phrased as statements about families of independent sets in certain uniform hypergraphs. In recent years, an important trend in the area has been to extend such classical results to the so-called `sparse random setting'. This line of research has recently culminated in the breakthroughs of Conlon and Gowers and of Schacht, who developed general tools for solving problems of this type. Although these two papers solved very similar sets of longstanding open problems, the methods used are very different from one another and have different strengths and weaknesses.

  In this paper, we provide a third, completely different approach to proving extremal and structural results in sparse random sets that also yields their natural `counting' counterparts. We give a structural characterization of the independent sets in a large class of uniform hypergraphs by showing that every independent set is almost contained in one of a small number of relatively sparse sets. We then derive many interesting results as fairly straightforward consequences of this abstract theorem. In particular, we prove the well-known conjecture of Kohayakawa, \L uczak, and R{\"o}dl, a probabilistic embedding lemma for sparse graphs. We also give alternative proofs of many of the results of Conlon and Gowers and Schacht, such as sparse random versions of Szemer\'edi's theorem, the Erd{\H o}s-Stone Theorem, and the Erd{\H o}s-Simonovits Stability Theorem, and obtain their natural `counting' versions, which in some cases are considerably stronger. For example, we show that for each positive $\beta$ and integer $k$, there are at most $\binom{\beta n}{m}$ sets of size~$m$ that contain no $k$-term arithmetic progression, provided that $m \ge Cn^{1-1/(k-1)}$, where $C$ is a constant depending only on $\beta$ and $k$. We also obtain new results, such as a sparse version of the Erd\H{o}s-Frankl-R\"odl Theorem on the number of $H$-free graphs and, as a consequence of the K\L R conjecture, we extend a result of R\"odl and Ruci\'nski on Ramsey properties in sparse random graphs to the general, non-symmetric setting. 
\end{abstract}

\maketitle

\section{Introduction}

A great many of the central questions in combinatorics fall into the following general framework: Given a finite set $V$ and a collection $\HH \subseteq \cP(V)$ of \emph{forbidden structures}, what can be said about sets $I \subseteq V$ that do not contain any member of $\HH$? For example, the celebrated theorem of Szemer{\'e}di~\cite{Sz} states that if $V = \{1, \ldots, n\}$ and $\HH$ is the collection of $k$-term arithmetic progressions in $\{1, \ldots, n\}$, then every set $I$ that contains no member of $\HH$ satisfies $|I| = o(n)$. The archetypal problem studied in extremal graph theory, dating back to the work of Tur{\'a}n~\cite{Turan} and Erd{\H o}s and Stone~\cite{ErSt}, is the problem of characterizing such sets $I$ when $V$ is the edge set of the complete graph on $n$ vertices and $\HH$ is the collection of copies of some fixed graph $H$ in $K_n$. In this setting, a great deal is known, not only about the maximum size of $I$ that contains no member of $\HH$, but also what the largest such sets look like, how many such sets there are, and what the structure of a typical such set is.

A collection $\HH \subseteq \cP(V)$ as above is usually referred to as a \emph{hypergraph} on the vertex set $V$ and any set $I \subseteq V$ that contains no element (\emph{edge}) of $\HH$ is called an \emph{independent set}. Therefore, one might say that a large part of extremal combinatorics is concerned with studying independent sets in various specific hypergraphs. We might add here that in many natural settings, such as the two mentioned above, the hypergraphs considered are \emph{uniform}, that is, all edges of $\HH$ have the same size.

Although it might at first seem somewhat artificial to study concrete questions in such an abstract setting, the past few years have proved that taking such a general approach can be highly beneficial. The recently-proved general transference theorems of Conlon and Gowers~\cite{CG} and Schacht~\cite{Sch} (see also~\cite{FRS}), which imply, among other things, sparse random analogues of the classical theorems of Szemer{\'e}di and of Erd{\H o}s and Stone, were stated in the language of hypergraphs. Roughly speaking, these transference theorems say the following: Let $\HH$ be a hypergraph whose edges are sufficiently `uniformly distributed'. Then the independence number of $\HH$ is `well-behaved' with respect to taking subhypergraphs induced by (sufficiently dense) random subsets of the vertex set. More precisely, given $p \in [0,1]$ and a finite set $V$, we shall write $V_p$ to denote the \emph{$p$-random subset} of $V$, that is, the random subset of $V$ in which each element of $V$ is included with probability $p$, independently of all other elements. We write $\alpha(\HH)$ and $v(\HH)$ to denote the size of the largest independent set and the number of vertices in a hypergraph $\HH$, respectively. The results of Conlon and Gowers~\cite{CG} and Schacht~\cite{Sch} imply, in particular, that if the distribution of the edges of some uniform hypergraph $\HH$ is sufficiently `balanced', then with probability tending to $1$ as $v(\HH) \to \infty$,
\[
\alpha\big( \HH[V(\HH)_p] \big) \le p \alpha(\HH) + o\big(p v(\HH)\big),
\]
provided that $p$ is sufficiently large.

In this work, we give an approximate structural characterization of the family of all independent sets in uniform hypergraphs whose edge distribution satisfies a certain natural boundedness condition. More precisely, we shall prove that the family $\I(\HH)$ of independent sets of such a hypergraph $\HH$ exhibits a certain clustering phenomenon. Our main result (Theorem~\ref{thm:main}, below) states that $\I(\HH)$ admits a partition into relatively few classes with the following property: all members of each class are essentially contained in a single `almost independent' subset of $V(\HH)$ (i.e., one which contains only a tiny proportion of all the edges of $\HH$). This somewhat abstract statement has surprisingly many deep and interesting consequences, some of which we list in the remainder of this section. We remark that Theorem~\ref{thm:main} was partly inspired by the work of Kleitman and Winston~\cite{KW}, who implicitly considered a statement of this type in the setting of graphs ($2$-uniform hypergraphs) and subsequently used it to bound the number of $n$-vertex graphs without a $4$-cycle. We also note that a result similar to Theorem~\ref{thm:main} was independently proved by Saxton and Thomason~\cite{SaTh}, who also use it to derive many of the statements that we present in Sections~\ref{sec:intro-APs}--~\ref{sec:KLR-intro}.

\subsection{The number of sets with no $k$-term arithmetic progression}

\label{sec:intro-APs}

The celebrated theorem of Szemer{\'e}di~\cite{Sz} says that for every $k \in \N$, the largest subset of $\{1, \ldots, n\}$ that contains no $k$-term arithmetic progression (AP) has $o(n)$ elements. It immediately follows that there are only $2^{o(n)}$ subsets of $\{1, \ldots, n\}$ with no $k$-term AP. Our first result can be viewed as a sparse analogue of this statement.

\begin{thm}
  \label{thm:Sz}
  For every positive $\beta$ and every $k \in \N$, there exist constants $C$ and $n_0$ such that the following holds. For every $n \in \N$ with $n \ge n_0$, if $m \ge Cn^{1-1/(k-1)}$, then there are at most
  \[
  \binom{\beta n}{m}
  \]
  $m$-subsets of $\{1, \ldots, n\}$ that contain no $k$-term AP.
\end{thm}

We shall deduce Theorem~\ref{thm:Sz} from our main theorem, Theorem~\ref{thm:main}, and a robust version of Szemer{\'e}di's theorem, see Section~\ref{sec:Sz}. The sparse random analogue of Szemer{\'e}di's theorem, proved by Schacht~\cite{Sch} and independently by Conlon and Gowers~\cite{CG}, follows as an easy corollary of Theorem~\ref{thm:Sz}. Following~\cite{CG}, we shall say that a set $A \subseteq \N$ is \emph{$(\delta,k)$-Szemer{\'e}di} if every subset $B \subseteq A$ with at least $\delta|A|$ elements contains a $k$-term AP. For the sake of brevity, let $[n] = \{1, \ldots, n\}$ and recall that $[n]_p$ denotes the $p$-random subset of $[n]$.

\begin{cor}
  \label{cor:Sz}
  For every $\delta \in (0,1)$ and every $k \in \N$, there exists a constant $C$ such that the following holds. If $p_n \ge Cn^{-1/(k-1)}$ for all sufficiently large $n$, then
  \[
  \lim_{n \to \infty} \Pr\big( \text{$[n]_{p_n}$ is $(\delta,k)$-Szemer{\'e}di} \big) = 1.
  \]
\end{cor}

We remark that Theorem~\ref{thm:Sz} and Corollary~\ref{cor:Sz} are both sharp up to the value of the constant $C$, see the discussion in Section~\ref{sec:Sz}, where both of these statements are proved.

Our main result has a variety of other applications in additive combinatorics, see for example~\cite{ABMS1,ABMS2} where, jointly with Alon, we used a much simpler version of it to count sum-free sets of fixed size in various Abelian groups and the set $[n]$. In Section~\ref{sec:Sz}, we shall mention two other applications: generalizations of Theorem~\ref{thm:Sz} to higher dimensions and to $k$-term APs whose common difference is of the form $d^r$. In each case, the random version (which was proved in~\cite{CG,Sch}) follows as an easy corollary.

\subsection{Tur{\'a}n's problem in random graphs}

\label{sec:turan-problem}

The famous theorem of Erd{\H o}s and Stone~\cite{ErSt} states that the maximum number of edges in an $H$-free graph on $n$ vertices, the \emph{Tur{\'a}n number for $H$}, denoted $\ex(n,H)$, satisfies
\begin{equation}
  \label{eq:exnH}
  \ex(n,H) = \left( 1 - \frac{1}{\chi(H) - 1} + o(1) \right) \binom{n}{2},
\end{equation}
where $\chi(H)$ is the chromatic number of $H$. The analogue of this theorem for the Erd{\H o}s-R{\'e}nyi random graph $G(n,p)$ was first studied by Babai, Simonovits, and Spencer~\cite{BaSiSp}, who proved that \emph{asymptotically almost surely} (a.a.s.~for short), i.e., with probability tending to~$1$ as $n \to \infty$, the largest triangle-free subgraph of $G(n,1/2)$ is bipartite, and by Frankl and R{\"o}dl~\cite{FrRo}, who proved that if $p \ge n^{-1/2 + \eps}$ then a.a.s.~the largest triangle-free subgraph of $G(n,p)$ has $pn^2/8 + o(pn^2)$ edges. The systematic study of the Tur\'an problem in $G(n,p)$ was initiated by Haxell, Kohayakawa, and \L uczak~\cite{HaKoLu95, HaKoLu96} and by Kohayakawa, \L uczak, and R\"odl~\cite{KLR97}, who posed the following problem. For a fixed graph $H$, determine necessary and sufficient conditions on a sequence $\p \in [0,1]^\N$ of probabilities such that, a.a.s.,
\begin{equation}
  \label{eq:exGnpH}
  \ex\big(G(n,p_n), H\big) = \left( 1 - \frac{1}{\chi(H)-1} + o(1) \right) \binom{n}{2} p_n,
\end{equation}
where $\ex(G,H)$ denotes the maximum number of edges in an $H$-free subgraph of $G$.

By considering a random $(\chi(H)-1)$-partition of the vertex set of $G(n,p)$, it is straightforward to show that the inequality $\ex\big(G(n,p), H \big) \ge \left(1 - \frac{1}{\chi(H) - 1} + o(1)\right) \binom{n}{2} p$ holds for every $p \in [0,1]$. On the other hand, if the number of copies of some subgraph $H' \subseteq H$ in $G(n,p)$ is much smaller than the number of edges in $G(n,p)$, then the converse inequality cannot hold, since one can make any graph $H$-free by removing from it one edge from each copy of $H'$. This observation motivates the notion of \emph{$2$-density} of $H$, denoted by $m_2(H)$, which is defined by
\begin{equation}
  \label{eq:m2H}
  m_2(H) = \max \left\{ \frac{e(H') - 1}{v(H') - 2} \colon H' \subseteq H \text{ with } v(H') \ge 3 \right\}.
\end{equation}
It now follows easily that for every graph $H$ with maximum degree at least $2$ and every $\delta \in \big(0, 1/(\chi(H)-1)\big)$, there exists a positive constant $c$ such that if $p_n \le c n^{-1/m_2(H)}$, then a.a.s.
\[
\ex\big(G(n,p_n), H\big) > \left( 1 - \frac{1}{\chi(H) - 1} + \delta \right) \binom{n}{2} p_n.
\]
It was conjectured by Haxell, Kohayakawa, and \L uczak~\cite{HaKoLu95} and Kohayakawa, \L uczak, and R{\"o}dl~\cite{KLR97} that the above simple argument, removing an arbitrary edge from each copy of $H'$ in $G(n,p)$, is the main obstacle that prevents~\eqref{eq:exGnpH} from holding asymptotically almost surely. The conjecture, often referred to as Tur{\'a}n's theorem for random graphs, has attracted considerable attention in the past fifteen years. Numerous partial results and special cases had been established by various researchers~\cite{Fu, Ge, GeScSt, HaKoLu95, HaKoLu96, KLR97, KRS04, SzVu} before the conjecture was finally proved by Conlon and Gowers~\cite{CG} (under the assumption that $H$ is \emph{strictly $2$-balanced}\footnote{A graph $H$ is $2$-balanced if the maximum in~\eqref{eq:m2H} is achieved with $H' = H$, that is, if $m_2(H) = \frac{e(H)-1}{v(H)-2}$. It is strictly $2$-balanced if $m_2(H) > m_2(H')$ for every proper subgraph $H' \subsetneq H$.}) and by Schacht~\cite{Sch}.

\begin{thm}
  \label{thm:Turan-Gnp}
  For every graph $H$ with $\Delta(H) \ge 2$ and every positive $\delta$, there exists a positive constant $C$ such that if $p_n \ge C n^{-1/m_2(H)}$, then a.a.s.
  \[
  \ex\big(G(n,p_n), H\big) \le \left( 1 - \frac{1}{\chi(H) - 1} + \delta \right) \binom{n}{2} p_n.
  \]  
\end{thm}

Our methods give yet another proof of Theorem~\ref{thm:Turan-Gnp}. In fact, we shall deduce from our main result, Theorem~\ref{thm:main}, a version of the general transference theorem of Schacht~\cite[Theorem~3.3]{Sch}, which easily implies Theorem~\ref{thm:Turan-Gnp} for such graphs $H$. Our version of Schacht's transference theorem, Theorem~\ref{thm:Sch-weak}, is stated and proved in Section~\ref{sec:extremal-results}. We then, in Section~\ref{sec:Turan}, use it to derive a natural generalization of Theorem~\ref{thm:Turan-Gnp} to $t$-uniform hypergraphs, Theorem~\ref{thm:t-Turan-Gnp}, which was also first proved in~\cite{CG} and~\cite{Sch}.

\begin{remark}
  In the original version of this paper, we only proved the results concerning $H$-free graphs under the additional assumption that $H$ is $2$-balanced. However, a simple modification of our method (permitting multiple edges in our hypergraphs, as in~\cite{SaTh}) allowed us to remove this condition. We would like to thank David Saxton for pointing this out. 
\end{remark}

Our methods also yield the following sparse random analogue of the famous stability theorem of Erd{\H o}s and Simonovits~\cite{ES1, ES2}, originally proved by Conlon and Gowers~\cite{CG} in the case when $H$ is strictly $2$-balanced and then extended to arbitrary $H$ by Samotij~\cite{Sam}, who adapted the argument of Schacht~\cite{Sch} for this purpose.

\begin{thm}
  \label{thm:stability-Gnp}
  For every graph $H$ with $\Delta(H) \ge 2$ and every positive $\delta$, there exist positive constants $C$ and $\eps$ such that if $p_n \ge Cn^{-1/m_2(H)}$, then a.a.s.~the following holds. Every $H$-free subgraph of $G(n,p_n)$ with at least
  \[
  \left(1 - \frac{1}{\chi(H)-1} -\eps\right) \binom{n}{2} p_n
  \]
  edges may be made $(\chi(H)-1)$-partite by removing from it at most $\delta n^2p_n$ edges.
\end{thm}

As with Theorem~\ref{thm:Turan-Gnp}, we shall in fact deduce Theorem~\ref{thm:stability-Gnp} from a more general statement, Theorem~\ref{thm:stability-random}, which is a version of the general transference theorem for stability results proved in~\cite{Sam}. Theorem~\ref{thm:stability-random} is stated and proved in Section~\ref{sec:stability-results}; in Section~\ref{sec:Turan}, we use it to derive Theorem~\ref{thm:stability-Gnp}.

\subsection{The typical structure of $H$-free graphs}

\label{sec:typical-structure}

Let $H$ be an arbitrary non-empty graph. For an integer $n$, denote by $f_n(H)$ the number of labelled $H$-free graphs on the vertex set $[n]$. Since every subgraph of an $H$-free graph is also $H$-free, it follows that $f_n(H) \ge 2^{\ex(n,H)}$. Erd{\H o}s, Frankl, and R{\"o}dl~\cite{ErFrRo} proved that this crude lower bound is in a sense tight, namely that
\begin{equation}
  \label{eq:fnH-upper}
  f_n(H) = 2^{\ex(n,H) + o(n^2)}.
\end{equation}
Our next result can be viewed as a `sparse version' of~\eqref{eq:fnH-upper}. Such a statement was already considered by \L uczak~\cite{Lu}, who derived it from the so-called K\L R conjecture, which we discuss in the next subsection. For integers $n$ and $m$ with $0 \le m \le \binom{n}{2}$, let $f_{n,m}(H)$ be the number of labelled $H$-free graphs on the vertex set $[n]$ that have exactly $m$ edges. The following theorem refines~\eqref{eq:fnH-upper} to $n$-vertex graphs with $m$ edges.

\begin{thm}
  \label{thm:ErFrRo-sparse}
 For every graph $H$ and every positive $\delta$, there exists a positive constant $C$ such that the following holds. For every $n \in \N$, if $m \ge Cn^{2-1/m_2(H)}$, then
  \[
  \binom{\ex(n,H)}{m} \le f_{n,m}(H) \le \binom{\ex(n,H) + \delta n^2}{m}.
  \]
\end{thm}

In fact, we shall deduce from our main result, Theorem~\ref{thm:main}, a `counting version' of the general transference theorem of Schacht~\cite[Theorem~3.3]{Sch}, which easily implies Theorem~\ref{thm:ErFrRo-sparse}. This `counting version' of Schacht's theorem (which refines and, in some respects, strengthens the main results of~\cite{CG,Sch}) is stated and proved in Section~\ref{sec:extremal-results}. We then use it to derive Theorem~\ref{thm:ErFrRo-sparse} in Section~\ref{sec:Turan-counting}. We remark that~\eqref{eq:fnH-upper} was refined in a different sense by Balogh, Bollob{\'a}s, and Simonovits~\cite{BaBoSi04}, who showed that $f_n(H) = 2^{\ex(n,H) + O(n^{2-c(H)})}$, where $c(H)$ is some positive constant, and also gave a very precise structural description of almost all $H$-free graphs. We would also like to point out that our proof of Theorem~\ref{thm:ErFrRo-sparse} does not use Szemer{\'e}di's regularity lemma, unlike the proof given in~\cite{Lu} or the proofs of Erd{\H o}s, Frankl, and R{\"o}dl~\cite{ErFrRo} and Balogh, Bollob{\'a}s, and Simonovits~\cite{BaBoSi04}.

The result of Erd{\H o}s, Frankl, and R{\"o}dl has, in some cases, a structural counterpart that significantly strengthens~\eqref{eq:fnH-upper}. For example, Erd{\H o}s, Kleitman, and Rothschild~\cite{ErKlRo} proved that \emph{almost all} triangle-free graphs are bipartite, that is, that with probability tending to $1$ as $n \to \infty$, a graph selected uniformly at random from the family of all triangle-free graphs on the vertex set $[n]$ is bipartite or, in other words (since clearly every bipartite graph is triangle-free), $f_n(K_3)$ is asymptotic to the number of bipartite graphs on the vertex set $[n]$. Extending this result, Osthus, Pr{\"o}mel, and Taraz~\cite{OsPrTa} proved that if $m \ge Cn^{3/2}\sqrt{\log n}$ for some $C > \sqrt{3}/4$, then almost all $n$-vertex triangle-free graphs with $m$ edges are bipartite. The corresponding result for $K_{r+1}$-free graphs was proved recently in~\cite{BMSW}. 

Our next result, which is a strengthening of Theorem~\ref{thm:ErFrRo-sparse}, is an approximate version of this statement for an arbitrary graph $H$. Such a statement was also considered by \L uczak~\cite{Lu}, who derived it from the K\L R conjecture. Following~\cite{Lu}, given a positive real $\delta$ and an integer $k$, let us say that a graph $G$ is $(\delta,k)$-partite if $G$ can be made $k$-partite by removing from it at most $\delta e(G)$ edges.

\begin{thm}
  \label{thm:H-free-structure}
  For every graph $H$ with $\chi(H) \ge 3$, and every positive $\delta$, there exists a positive constant $C$ such that the following holds. If $m \ge Cn^{2 - 1/m_2(H)}$, then almost all $H$-free graphs with $n$ vertices and $m$ edges are $\big(\delta,\chi(H)-1\big)$-partite.
\end{thm}

As with Theorem~\ref{thm:ErFrRo-sparse}, we shall in fact deduce Theorem~\ref{thm:H-free-structure} from a `counting version' of the general transference theorem for stability results proved in~\cite{Sam}. Our version of it, Theorem~\ref{thm:stability-counting}, is stated and proved in Section~\ref{sec:stability-results}. In Section~\ref{sec:Turan-counting}, we use it to derive Theorem~\ref{thm:H-free-structure}. Once again, our proof does not use the regularity lemma, unlike that in~\cite{Lu}. Finally, we would like to mention that, as observed by \L uczak~\cite{Lu}, Theorem~\ref{thm:H-free-structure} has the following elegant corollary.

\begin{cor}
  \label{cor:H-free-prob}
  For every graph $H$  with $\chi(H) \ge 3$ and every positive $\eps$, there exist positive constants $C$  and $n_0$ such that the following holds. For every $n \in \N$ with $n \ge n_0$ and every $m \in \N$ with $Cn^{2-1/m_2(H)} \le m \le n^2/C$,
  \begin{equation}\label{eq:cor:luczak}
  \left(\frac{\chi(H) - 2}{\chi(H) - 1} - \eps\right)^m \le \, \Pr\big(G_{n,m} \nsupseteq H\big) \le \left(\frac{\chi(H) - 2}{\chi(H) - 1} + \eps\right)^m,
  \end{equation}
  where $G_{n,m}$ is a uniformly selected random $n$-vertex graph with $m$ edges.
\end{cor}

Note that~\eqref{eq:cor:luczak} does not hold if $m$ is too large; for example, if $m > n^2/4$ then $\Pr\big(G_{n,m} \nsupseteq K_3 \big) = 0$. We remark that a great deal more is known about the structure of a typical $H$-free graph (drawn uniformly at random from the set of all $n$-vertex $H$-free graphs), see~\cite{BaBoSi09} and the references therein for more details. 

\subsection{The K\L R conjecture}

\label{sec:KLR-intro}

The celebrated Szemer{\'e}di regularity lemma~\cite{Sz78}, which is considered to be one of the most important and powerful tools in extremal graph theory, says that the vertex set of every graph may be divided into a bounded number of parts of approximately the same size in such a way that most of the bipartite subgraphs induced between pairs of parts of the partition satisfy a certain pseudo-randomness condition termed \emph{$\eps$-regularity}. The strength of the regularity lemma lies in the fact that it may be combined with the so-called \emph{embedding lemma} to show that a graph contains particular subgraphs. The combination of the regularity and embedding lemmas allows one to prove many well-known theorems in extremal graph theory, such as the theorem of Erd{\H o}s and Stone~\cite{ErSt} and the stability theorem of Erd{\H o}s and Simonovits~\cite{ES1,ES2}, both mentioned in Section~\ref{sec:turan-problem}.

For sparse graphs, that is, $n$-vertex graphs with $o(n^2)$ edges, the original version of the regularity lemma is vacuous since if the vertex set of a sparse graph is partitioned into a bounded number of parts, then all induced bipartite subgraphs thus obtained are trivially $\eps$-regular, provided that $n$ is sufficiently large. However, it was independently observed by Kohayakawa~\cite{Ko97} and R{\"o}dl (unpublished) that the notion of $\eps$-regularity may be extended in a meaningful way to graphs with density tending to zero. Moreover, with this more general notion of regularity, they were also able to prove an associated regularity lemma which applies to a large class of sparse graphs, including (a.a.s.) the random graph $G(n,p)$.

Given a $p \in [0,1]$ and a positive $\eps$, we say that a bipartite graph between sets $V_1$ and $V_2$ is \emph{$(\eps,p)$-regular} if for every $W_1 \subseteq V_1$ and $W_2 \subseteq V_2$ with $|W_1| \ge \eps|V_1|$ and $|W_2| \ge \eps |V_2|$, the density $d(W_1, W_2)$ of edges between $W_1$ and $W_2$ satisfies
\[
\big| d(W_1, W_2) - d(V_1, V_2) \big| \le \eps p.
\]
A partition of the vertex set of a graph into $r$ parts $V_1, \ldots, V_r$ is said to be $(\eps, p)$-regular if $\big| |V_i| - |V_j| \big| \le 1$ for all $i$ and $j$ and for all but at most $\eps r^2$ pairs $(V_i, V_j)$, the graph induced between $V_i$ and $V_j$ is $(\eps, p)$-regular. The class of graphs to which the Kohayakawa-R{\"o}dl regularity lemma applies are the so-called upper-uniform graphs. Given positive $\eta$ and $K$, we say that an $n$-vertex graph $G$ is \emph{$(\eta, p, K)$-upper-uniform} if for all $W \subseteq V(G)$ with $|W| \ge \eta n$, the density of edges within $W$ satisfies $d(W) \le Kp$. This condition is satisfied by many natural classes of graphs, including (a.a.s.) all subgraphs of random graphs of density $p$. The sparse regularity lemma of Kohayakawa~\cite{Ko97} and R{\"o}dl says the following.

\begin{SpSzRL}
  For all positive $\eps$, $K$, and $r_0$, there exist a positive constant $\eta$ and an integer $R$ such that for every $p \in [0,1]$, the following holds. Every $(\eps,p,K)$-upper-uniform graph with at least $r_0$ vertices admits an $(\eps,p)$-regular partition of its vertex set into $r$ parts, for some $r \in \{r_0, \ldots, R\}$.
\end{SpSzRL}

We remark that a version of this theorem avoiding the need for the upper-uniformity assumption was recently proved by Scott~\cite{Sc11}.

The aforementioned embedding lemma roughly says that if we start with an arbitrary graph $H$, replace its vertices by large independent sets and its edges by $\eps$-regular bipartite graphs with density much larger than $\eps$, then this blown-up graph will contain a copy of $H$. To make it more precise, let $H$ be a graph on the vertex set $\{1, \ldots, v(H)\}$, let $\eps$ and $p$ be as above, and let $n$ and $m$ be integers satisfying $0 \le m \le n^2$. Let us denote by $\G(H,n,m,p,\eps)$ the collection of all graphs $G$ constructed in the following way. The vertex set of $G$ is a disjoint union $V_1 \cup \ldots \cup V_{v(H)}$ of sets of size $n$, one for each vertex of $H$. For each edge $\{i,j\}$ of $H$, we add to $G$ an $(\eps,p)$-regular bipartite graph with $m$ edges between the sets $V_i$ and $V_j$. These are the only edges of $G$. With this notation in hand, we can state the embedding lemma. Given any graph $G$ as above, we define \emph{canonical copies of $H$} to be all copies of $H$ in $G$ in which (the image of) each vertex $i \in V(H)$ lies in the set $V_i \subseteq V(G)$.

\begin{EL}
  For every graph $H$ and every positive $d$, there exist a positive $\eps$ and an integer $n_0$ such that for every $n$ and $m$ with $n \ge n_0$ and $m \ge dn^2$, every $G \in \G(H,n,m,1,\eps)$ contains a canonical copy of $H$.
\end{EL}

One might hope that a similar statement holds when one replaces $1$ by an arbitrary $p$ and the assumption $m \ge dn^2$ by $m \ge pdn^2$, even if $p$ is a decreasing function of $n$. However, for an arbitrary function $p$, this is too much to hope for. Indeed, consider the random `blow-up' of $H$, that is, the random graph $G$ obtained from $H$ by replacing each vertex of $H$ by an independent set of size $n$ and each edge of $H$ by a random bipartite graph with $pn^2$ edges. With high probability, the number of canonical copies of $H$ in $G$ will be about $p^{e(H)}n^{v(H)}$ and hence if $p^{e(H)}n^{v(H)} \ll pn^2$, then one can remove all copies of $H$ from $G$ by deleting a tiny proportion of all edges. Since in the above argument one may replace $H$ with an arbitrary subgraph $H' \subseteq H$, it follows easily\footnote{Note that we also replace $p$ with some $p' = (1 + o(1))p$, and that the removal of $o(pn^2)$ edges does not affect the $\eps$-regularity conditions.} that if $p \ll n^{-1/m_2(H)}$, then there are graphs in $\G(H,n,pn^2,p,\eps)$ that do not contain any canonical copies of~$H$.

\enlargethispage{\baselineskip}

As in the case of Tur{\'a}n's theorem for random graphs, see Section~\ref{sec:turan-problem}, one might still hope that if $p \ge Cn^{-1/m_2(H)}$ for some large constant $C$, then the natural sparse analogue of the embedding lemma discussed above holds. However, it was observed by \L uczak (see~\cite{GeSt,KoRo}) that, somewhat surprisingly, for any graph $H$ which contains a cycle and any function $p$ satisfying $p = o(1)$, there are graphs in $\G(H,n,pn^2,p,\eps)$ with no canonical copy of $H$. Nevertheless, it still seemed likely that such atypical graphs comprise so tiny a proportion of $\G(H,n,m,p,\eps)$ that they do not appear in $G(n,p)$ asymptotically almost surely. 

This was formalized in the following conjecture of Kohayakawa, \L uczak, and R{\"o}dl~\cite{KLR97}, usually referred to as the \emph{K\L R conjecture}. Given a graph $H$, integers $m$ and $n$, a $p \in [0,1]$, and a positive $\eps$, let $\G^*(H,n,m,p,\eps)$ denote the collection of graphs in $\G(H,n,m,p,\eps)$ that contain no canonical copy of $H$. We will prove the conjecture in Section~\ref{sec:KLR}.

\begin{thm}[The K\L R conjecture]\label{thm:KLR}
 For every graph $H$ and every positive $\beta$, there exist positive constants $C$, $n_0$, and $\eps$ such that the following holds. For every $n \in \N$ with $n \ge n_0$ and $m \in \N$ with $m \ge Cn^{2-1/m_2(H)}$, 
  \[
  \big| \G^*(H,n,m,m/n^2,\eps) \big| \le \beta^m \binom{n^2}{m}^{e(H)}.
  \]
\end{thm}

The K\L R conjecture has been one of the central open questions in extremal graph theory and has attracted substantial attention from many researchers over the past fifteen years. It has been verified in several special cases. It is easy to see that it holds for all graphs $H$ which do not contain a cycle. The cases $H = K_3$, $K_4$, and $K_5$ were resolved in~\cite{KLR96}, \cite{GePrScStTa}, and~\cite{GeScSt}, respectively. The case $H = C_\ell$ has also been resolved, but here the history is somewhat more complex. A proof under some extra technical assumptions was given in~\cite{KoKr}. Those extra assumptions were later removed in~\cite{GeKoRoSt} and, independently, in \cite{Be}. We remark here that in parallel to this work, Conlon, Gowers, Samotij, and Schacht~\cite{CoGoSaSc} have proved a sparse analogue of the \emph{counting lemma} for subgraphs of the random graph $G(n,p)$, which may be viewed as a version of the K\L R conjecture that is stronger in some aspects and weaker in other aspects. 

It is well-known that Theorem~\ref{thm:KLR} easily implies Tur{\'a}n's theorem for random graphs, Theorem~\ref{thm:Turan-Gnp}, and also its stability version, Theorem~\ref{thm:stability-Gnp}. In fact, this was the original motivation behind the K\L R conjecture, see~\cite{KLR97}. Moreover, it was proved by \L uczak~\cite{Lu} that Theorem~\ref{thm:KLR} implies Theorems~\ref{thm:ErFrRo-sparse} and~\ref{thm:H-free-structure}. The work of Conlon and Gowers~\cite{CG} and Schacht~\cite{Sch} (see also~\cite{Sam}), as well as this work, have shown that one does not need to appeal to the sparse regularity lemma and to the K\L R conjecture in order to prove such extremal statements in random graphs. Nevertheless, there are still many beautiful corollaries of the conjecture that cannot (yet) be proved by other means. For discussion and derivation of some of them, we refer the reader to~\cite{CoGoSaSc}. Here, we present only one corollary of the K\L R conjecture, the threshold for asymmetric Ramsey properties of random graphs, which does not follow from the version of the conjecture proved in~\cite{CoGoSaSc}. The deduction of this result from the K\L R conjecture is essentially due to Kohayakawa and Kreuter~\cite{KoKr}. 

\subsection{Ramsey properties of random graphs}

Let $H$ be a fixed graph and let $r$ be a positive integer. For an arbitrary graph $G$, we write $G \to (H)_r$ if every $r$-coloring of the edges of $G$ contains a monochromatic copy of $H$. It follows from the classical result of Ramsey~\cite{Ra} that $K_n \to (H)_r$, provided that $n$ is sufficiently large. Ramsey properties of random graphs were first investigated by Frankl and R{\"o}dl~\cite{FrRo} and since then much effort has been devoted to their study. Most notably, R\"odl and Ruci\'nski~\cite{RoRu93, RoRu95} established the following general threshold result.

\begin{thm}
  \label{thm:RoRu}
    For every graph $H$ that is not a forest, and every positive integer $r$, there exist positive constants $c$ and $C$ such that
  \[
  \lim_{n \to \infty} \Pr\big( G(n,p_n) \to (H)_r \big) =
  \begin{cases}
    1 & \text{if $p_n \ge Cn^{-1/m_2(H)}$}, \\
    0 & \text{if $p_n \le cn^{-1/m_2(H)}$}.
  \end{cases}
  \]
\end{thm}

In the above discussion, a copy of the same graph $H$ is forbidden in each of the $r$ color classes. A natural generalization of Theorem~\ref{thm:RoRu} would determine thresholds for so-called \emph{asymmetric Ramsey properties}. For any graphs $G$, $H_1, \ldots, H_r$, we write $G \to (H_1, \ldots, H_r)$ if for every coloring of the edges of $G$ with colors $1, \ldots, r$, there exists, for some $i \in [r]$, a copy of $H_i$ all of whose edges have color $i$. In the context of asymmetric Ramsey properties of random graphs, the following generalization of the $2$-density $m_2(\cdot)$ was introduced in~\cite{KoKr}. For two graphs $H_1$ and $H_2$, define\footnote{To motivate this definition, set $p = n^{-1 / m_2(H_1,H_2)}$ and observe that the edges of $G(n,p)$ which are contained in a copy of each subgraph $H_1' \subseteq H_1$ have density roughly $n^{-1/m_2(H_2)}$.} 
\begin{equation}
  \label{eq:m2H2H1}
  m_2(H_1, H_2) = \max \left\{ \frac{e(H_1') }{v(H_1') - 2 + 1/m_2(H_2)} \,\colon H_1' \subseteq H_1 \text{ with } v(H_1') \ge 3 \right\}.
\end{equation}
Kohayakawa and Kreuter~\cite{KoKr} formulated the following conjecture and proved it in the case when all $H_i$ are cycles.

\begin{conj}
  \label{conj:KoKr}
  Let $H_1, \ldots, H_r$ be graphs with $1 < m_2(H_r) \le \ldots \le m_2(H_1)$. Then there exist constants $c$ and $C$ such that
  \[
  \lim_{n \to \infty} \Pr\big( G(n,p_n) \to (H_1, \ldots, H_r) \big) =
  \begin{cases}
    1 & \text{if $p_n \ge Cn^{-1/m_2(H_1, H_2)}$}, \\
    0 & \text{if $p_n \le cn^{-1/m_2(H_1, H_2)}$}.
  \end{cases}
  \]
\end{conj}

More accurately, the above conjecture was stated in~\cite{KoKr} only in the case $r=2$, but the above generalization is quite natural.\footnote{To see why the graphs $H_3,\ldots,H_r$ do not appear in the threshold, replace each of $H_2,\ldots,H_r$ by the disjoint union $H' = H_2 \cup \cdots \cup H_r$, and note that $m_2(H') = m_2(H_2)$, see~\cite{MaSkSpSt}.} There had been little progress on Conjecture~\ref{conj:KoKr} until quite recently, when the $0$-statement was proved by Marciniszyn, Skokan, Sp\"ohel, and Steger~\cite{MaSkSpSt} in the case where all of the $H_i$ are cliques, and the $1$-statement in the case $r = 2$ was established\footnote{In their concluding remarks, the authors of~\cite{KoSchSp} moreover claim that  their method can be extended to the setting with more than two colours, using  ideas from~\cite{RoRu95}.} by Kohayakawa, Schacht, and Sp\"ohel~\cite{KoSchSp} under very mild extra assumptions on $H_1$ and $H_2$. It was observed in~\cite[Theorem~31]{MaSkSpSt} that, using Theorem~\ref{thm:KLR}, the approach of Kohayakawa and Kreuter~\cite{KoKr}, which employs the sparse regularity lemma, can be adapted to yield a proof of the $1$-statement in Conjecture~\ref{conj:KoKr} for the following class of graphs.

\begin{thm}
  \label{thm:Ramsey-Gnp}
  Let $H_1, \ldots, H_r$ be graphs with $1 < m_2(H_r) \le \ldots \le m_2(H_1)$ and such that $H_1$ is strictly $2$-balanced. Then there exists a constant $C$ such that if $p_n \ge Cn^{-1/m_2(H_1,H_2)}$, then a.a.s.
  \[
  G(n,p_n) \to (H_1, \ldots, H_r).
  \]
\end{thm}

For the deduction of Theorem~\ref{thm:Ramsey-Gnp} from Theorem~\ref{thm:KLR}, see~\cite{KoKr} and~\cite[Section~4]{MaSkSpSt}.

\subsection{Outline of the paper}

The remainder of this paper is organized as follows. In Section~\ref{sec:main}, we state and discuss our main result, Theorem~\ref{thm:main}, which we then prove in Section~\ref{sec:proof-thm-main}. In Section~\ref{sec:Sz}, we discuss the applications of Theorem~\ref{thm:main} in the context of subsets of $[n]$ with no $k$-term arithmetic progressions. In particular, we prove Theorem~\ref{thm:Sz} and use it to derive Corollary~\ref{cor:Sz}. In Section~\ref{sec:extremal-results}, we prove two versions of the general transference theorem of Schacht~\cite[Theorem~3.3]{Sch} (obtained independently, in a slightly different form, by Conlon and Gowers~\cite{CG}) -- a `random' version suited for extremal problems in sparse random discrete structures and its `counting' counterpart that generalizes Theorem~\ref{thm:Sz}. In Section~\ref{sec:stability-results}, we prove `random' and `counting' versions of the general stability result of Conlon and Gowers~\cite{CG} in a form that is easily comparable with~\cite[Theorem~3.4]{Sam}. In Section~\ref{sec:Turan}, we discuss several applications of Theorem~\ref{thm:main} in the context of the Tur{\'a}n problem in sparse random graphs. In particular, using the results of Sections~\ref{sec:extremal-results} and~\ref{sec:stability-results} we give new proofs of the sparse random analogues (stated above) of the classical theorems of Erd{\H o}s and Stone, and Erd{\H o}s and Simonovits, see Section~\ref{sec:turan-problem}. In Section~\ref{sec:Turan-counting}, we discuss applications of Theorem~\ref{thm:main} to the problem of describing the typical structure of a sparse graph without a forbidden subgraph. In particular, we prove sparse analogues of classical theorems of Erd{\H o}s, Frankl, and R{\"o}dl and Erd{\H o}s, Kleitman, and Rothschild, see Section~\ref{sec:typical-structure}. Finally, in Section~\ref{sec:KLR}, we use Theorem~\ref{thm:main} to prove the K\L R conjecture for every graph~$H$. 

\section{The Main Theorem}

\label{sec:main}

In this section, we present the main result of this paper, Theorem~\ref{thm:main}, which gives a structural characterization of the collection of all independent sets in a large class of uniform hypergraphs. Let us stress here that all of the hypergraphs we consider are allowed to have multiple edges; moreover, we shall always count edges with multiplicities. 

We start with an important definition. Recall that a family of sets $\F \subseteq \cP(V)$ is called \emph{increasing} (or an \emph{upset}) if it is closed under taking supersets, that is, if for every $A, B \subseteq V$, $A \in \F$ and $A \subseteq B$ imply that $B \in \F$.

\begin{defn}
  \label{defn:Feps-dense}
  Let $\HH$ be a uniform hypergraph with vertex set $V$, let $\F$ be an increasing family of subsets of $V$ and let $\eps \in (0,1]$. We say that $\HH$ is \emph{$(\F,\eps)$-dense} if
  \[
  e(\HH[A]) \ge \eps e(\HH)
  \]
  for every $A \in \F$.
\end{defn}

A moment of thought reveals that for an arbitrary hypergraph $\HH$ and $\eps \in (0,1]$, it is extremely simple to find families $\F \subseteq \cP(V(\HH))$ for which $\HH$ is $(\F, \eps)$-dense. To this end, let
\[
\F_{\eps} = \big\{ A \subseteq V(\HH) \colon e(\HH[A]) \ge \eps e(\HH) \big\}
\]
and note that $\F_{\eps}$ is increasing and $\HH$ is $(\F_{\eps}, \eps)$-dense. In fact, the families $\F$ for which $\HH$ is $(\F, \eps)$-dense are precisely all increasing subfamilies of $\F_\eps$. 

In this work, we will be interested in upsets that admit a much more `constructive' description than that of $\F_\eps$. Many such families arise naturally in the study of extremal and structural problems in combinatorics. For example, consider the $k$-uniform hypergraph $\HH_1$ on the vertex set $[n]$ whose edges are all $k$-term arithmetic progressions in $[n]$ and let $\F_1$ be the collection of all subsets of $[n]$ with at least $\delta n$ elements. Clearly, $\F_1$ is an upset and it follows from the famous theorem of Szemer{\'e}di~\cite{Sz} that $\HH_1$ is $(\F_1, \eps)$-dense for some positive $\eps$ depending only on $\delta$ and $k$, see Section~\ref{sec:Sz}. Similarly, consider the $3$-uniform hypergraph $\HH_2$ on the vertex set $E(K_n)$ whose edges are edge sets of all copies of $K_3$ in the complete graph $K_n$ and let $\F_2$ be the family of all $n$-vertex graphs (subgraphs of $K_n$) with at least $(1/2 - \eps)\binom{n}{2}$ edges such that every $2$-coloring of its vertices yields at least $\delta n^2$ monochromatic edges. Again, $\F_2$ is increasing and it follows from the stability theorem of Erd{\H o}s and Simonovits~\cite{ES1,ES2} and the triangle removal lemma of Ruzsa and Szemer{\'e}di~\cite{RuSz} that $\HH_2$ is $(\F_2, \eps)$-dense, provided that $\eps$ is sufficiently small as a function of $\delta$.

Our main result roughly says the following. If $\HH$ is a uniform hypergraph that is $(\F, \eps)$-dense for some family $\F$ and whose edge distribution satisfies certain natural boundedness conditions, then the collection $\I(\HH)$ of all independent sets in $\HH$ admits a partition into relatively few classes such that all independent sets in one class are essentially contained in a single set $A \not\in \F$. Before we state the result, we first need to quantify the above boundedness conditions for the edge distribution of a hypergraph. Given a hypergraph $\HH$, for each $T \subseteq V(\HH)$, we define\footnote{We emphasize that if $\HH$ has multiple edges, then $\{ e \in \HH \colon T \subseteq e \}$ should be thought of as a multi-set. In other words, $\deg_\HH(T)$ is the number of edges of $\HH$, counted with multiplicities, which contain $T$.}
\[
\deg_\HH(T) = | \{ e \in \HH \colon T \subseteq e \} |,
\]
and let
\[
\Delta_\ell(\HH) = \max \big\{ \deg_\HH(T) \colon T \subseteq V(\HH) \text{ and } |T| = \ell \big\}.
\]
Recall that $\I(\HH)$ denotes the family of all independent sets in $\HH$. The following theorem is our main result.

\begin{thm}
  \label{thm:main}
  For every $k \in \N$ and all positive $c$ and $\eps$, there exists a positive constant $C$ such that the following holds. Let $\HH$ be a $k$-uniform hypergraph and let $\F \subseteq \cP(V(\HH))$ be an increasing family of sets such that $|A| \ge \eps v(\HH)$ for all $A \in \F$. Suppose that $\HH$ is $(\F,\eps)$-dense and $p \in (0,1)$ is such that, for every $\ell \in [k]$,
  \[
  \Delta_\ell(\HH) \le c \cdot p^{\ell-1} \frac{e(\HH)}{v(\HH)}.
  \]
  Then there exists a family $\S \subseteq \binom{V(\HH)}{\le Cp \cdot v(\HH)}$ and functions $f \colon \S \to \ol{\F}$ and $g \colon \I(\HH) \to \S$ such that for every $I \in \I(\HH)$,
  \[
  g(I) \subseteq I \qquad \text{and} \qquad I \setminus g(I) \subseteq f( g(I) ).
  \]
\end{thm}

Roughly speaking, if $\HH$ satisfies certain technical conditions, then each independent set $I$ in $\HH$ can be labelled with a small subset $g(I)$ in such a way that all sets labelled with some $S \in \S$ are essentially contained in a single set $f(S)$ that contains very few edges of $\HH$. We remark that the constant $C$ in the theorem has only a polynomial dependence on $\eps$. Unfortunately, however, in most of our applications $\eps$ will have a tower-type dependence on some other parameter.

Theorem~\ref{thm:main} will be proved in Section~\ref{sec:proof-thm-main}. We end this section with a short informal discussion of its consequences. As we have already mentioned, Theorem~\ref{thm:main} combined with some classical extremal results on discrete structures has strikingly strong implications. Let us briefly explain why this is so. Many classical extremal problems ask for an estimate on the number of independent sets (of a certain size) in some auxiliary uniform hypergraph. If applicable, Theorem~\ref{thm:main} implies that all such independent sets are almost contained in one of very few sets that are almost independent, that is, contain a small number of copies of some forbidden substructure. If we know a good characterization of sets that are almost independent in the above sense, which is often the case, we can easily obtain an upper bound on the number of independent sets. For example, consider the problem of counting subsets of $[n]$ with no $k$-term AP and recall the definition of $\HH_1$ and $\F_1$ from the beginning of this section. Theorem~\ref{thm:main}, applied to this pair, implies that every subset of $[n]$ with no $k$-term AP is essentially contained in one of at most $\binom{n}{O(n^{1-1/(k-1)})}$ sets of size at most $\delta n$ each, where $\delta$ is an arbitrarily small positive constant. This easily implies that if $m \gg n^{1-1/(k-1)}$, then there are at most $\binom{2\delta n}{m}$ sets of size $m$ with no $k$-term AP. For more details, we refer the reader to Section~\ref{sec:Sz}.

\section{Proof of the main theorem}

\label{sec:proof-thm-main}

In this section, we shall prove Theorem~\ref{thm:main}. The main ingredient in the proof is the following proposition, which (roughly) says that Theorem~\ref{thm:main} holds in the special case when $\F$ is the family of all subsets of $V(\HH)$ with at least $(1-\delta)v(\HH)$ elements. Theorem~\ref{thm:main} follows by applying Proposition~\ref{prop:main} a constant number of times. 

\begin{prop}
  \label{prop:main}
  For every integer $k$ and positive $c$, there exists a positive $\delta$ such that the following holds. Let $p \in (0,1)$ and suppose that $\HH$ is a $k$-uniform hypergraph such that, for every $\ell \in [k]$,
  \[
  \Delta_\ell(\HH) \le c \cdot p^{\ell-1} \frac{e(\HH)}{v(\HH)}.
  \]
  Then there exist a family $\S \subseteq \binom{V(\HH)}{\le (k-1)p \cdot v(\HH)}$ and functions $f_0 \colon \S \to \cP(V(\HH))$ and $g_0 \colon \I(\HH) \to \S$ such that for every $I \in \I(\HH)$,
  \[
  g_0(I) \subseteq I \subseteq f_0(g_0(I)) \cup g_0(I) \qquad \text{and} \qquad \big| f_0( g_0(I) ) \big| \le (1-\delta)v(\HH).
  \]
  Moreover, if for some $I, I' \in \I(\HH)$, $g_0(I) \subseteq I'$ and $g_0(I') \subseteq I$, then $g_0(I) = g_0(I')$.
\end{prop}

The final line of Proposition~\ref{prop:main} states that the labelling function $g_0$ exhibits a certain consistency. This property of $g_0$, which may look somewhat puzzling, will be crucial in the proof of Theorem~\ref{thm:main}.

In order to prove Proposition~\ref{prop:main}, given an independent set $I \in \I(\HH)$, we shall construct a sequence $(B_{k-1}, \ldots, B_q)$ of subsets of $I$ with $|B_{k-1}|, \ldots, |B_q| \le p v(\HH)$, for some $q \in [k-1]$, and use it to define a sequence $(\HH_{k-1}, \ldots, \HH_{r})$, where $r \in \{q, q+1\}$, of hypergraphs such that the following holds for each $i \in \{r, \ldots, k-1\}$:
\begin{enumerate}[(a)]
\item
  $\HH_i$ is an $i$-uniform hypergraph on the vertex set $V(\HH)$,
\item
  $I$ is an independent set in $\HH_i$,
\item
  \label{item:DeltaHHi}
  $\Delta_1(\HH_i) \le O\big(e(\HH_i) / v(\HH_i)\big)$, and
\item
  \label{item:eHHi}
  $e(\HH_i) \ge \Omega(p^{k-i} e(\HH))$.
\end{enumerate}
We shall be able to do it in such a way that in the end, there will be a set $A \subseteq V(\HH)$ of size at most $(1-\delta)v(\HH)$ such that the remaining elements of $I$ (i.e., the set $I \setminus S$, where $S = B_k \cup \cdots \cup B_q$) must all lie inside $A$. If $r = 1$, then we will simply let $A$ be the set of non-edges of the $1$-uniform hypergraph $\HH_1$; in this case, the upper bound on $|A|$ will follow from~(\ref{item:DeltaHHi}) and~(\ref{item:eHHi}). If $r > 1$, then we will obtain an appropriate $A$ while trying (and failing) to construct the hypergraph $\HH_{r-1}$ using the hypergraph $\HH_r$ and the set $B_r$. Crucially, this set $A$ will depend solely on $S$, that is, if for some pair $I, I' \in \I(\HH)$ our procedure generates $(S, A)$ and $(S', A')$, respectively, and if $S = S'$, then also $A = A'$. This will allow us to set $g_0(I) = S$ and $f_0(S) = A$.

\subsection{The Algorithm Method}

\label{sec:WSAlg}

For the remainder of this section, let us fix $k$, $c$, $p$, and $\HH$ as in the statement of Proposition~\ref{prop:main}. Without loss of generality, we may assume that $c \ge 1$. Let $I$ be an independent set in $\HH$. We shall describe a procedure of choosing the sets $B_i \subseteq I$ and constructing the hypergraphs $\HH_i$ as above. This procedure, which we shall term the \emph{Scythe Algorithm}, lies at the heart of the proof of Proposition~\ref{prop:main}.

The general strategy used in the Scythe Algorithm, that of selecting a small set $S$ of high-degree vertices and using it to define a set $A$ such that $S \subseteq I \subseteq A \cup S$, dates back to the work of Kleitman and Winston~\cite{KW}, who used it to bound the number of independent sets in graphs satisfying the following local density condition: all sufficiently large vertex sets induce subgraphs with many edges. Recently, Balogh and Samotij~\cite{BS1,BS2} refined the ideas of Kleitman and Winston and obtained a bound on the number of independent sets in uniform hypergraphs satisfying a similar local density condition. Even more recently, Alon, Balogh, Morris and Samotij~\cite{ABMS1} used similar ideas to bound the number of independent sets in `almost linear' $3$-uniform hypergraphs satisfying a more general density condition termed $(\alpha,\B)$-stability, see Definition~\ref{defn:aB-stable}. Here, we combine, generalize, and refine all of the above approaches and make them work in the general setting of $(\F, \eps)$-dense uniform hypergraphs.

At each step of the Scythe Algorithm, we shall order the vertices of a certain subhypergraph of $\HH$ with respect to their degrees in that subhypergraph. For the sake of brevity and clarity of the presentation, let us make the following definition.

\begin{defn}[Max-degree order]
  Given a hypergraph $\G$, we define the \emph{max-degree order} on $V(\G)$ as follows:
  \begin{enumerate}
  \item
    \label{item:mdo-a}
    Fix an arbitrary total ordering of $V(\G)$.
  \item
    For each $j \in \{1, \ldots, v(\G)\}$, let $u_j$ be the maximum-degree vertex in the hypergraph $\G\big[ V(\G) \setminus \{u_1, \ldots, u_{j-1}\} \big]$; ties are broken by giving preference to vertices which come earlier in the order chosen in (\ref{item:mdo-a}).
  \item
    The max-degree order on $V(\G)$ is $(u_1,\ldots,u_{v(\G)})$. 
  \end{enumerate}
  Finally, we write $W(u)$ to denote the initial segment of the max-degree order on $V(\G)$ that ends with $u$, i.e., for every $j$, we let $W(u_j) = \{u_1, \ldots, u_j\}$.
\end{defn}
We remark here that the only property of the max-degree order that will be important for us is that for every $j \in \{1, \ldots, v(\G)\}$, the degree of the vertex $u_j$ in the hypergraph $\G[V(\G) \setminus W(u_{j-1})]$ is at least as large as the average degree of this hypergraph.

We next define the numbers $\Delta_\ell^i$, where $1 \le \ell \le i \le k$, which will play a crucial role in the description and the analysis of the algorithm.

\begin{defn}\label{def:delta:is}
  For every $\ell \in [k]$, let $\Delta_\ell^k = \Delta_\ell(\HH)$ and for all $i \in [k-1]$ and $\ell \in [i]$, let
  \begin{equation}
    \label{def:delta}
    \Delta_\ell^i = \max \left\{ 2 \cdot \Delta_{\ell+1}^{i+1}, \, p \cdot \Delta_{\ell}^{i+1} \right\}.
  \end{equation}
\end{defn}

We use the numbers $\Delta_\ell^i$ to define the following families of sets with high degree.

\begin{defn}\label{def:Msets}
  Given an $i \in [k]$, an $i$-uniform hypergraph $\G$ and an $\ell \in [i]$, let
  \[
  M_\ell^i(\G) = \left\{ T \in \binom{V(\G)}{\ell} \colon \deg_\G(T) \ge \frac{\Delta_\ell^i}{2} \right\}.
  \]
\end{defn}

Let $b = p v(\HH)$ and for each $i \in [k]$, let $c_i = (ck2^{k+1})^{i-k}$.

\begin{Prop}
  The key properties that we would like the constructed hypergraph $\HH_i$ to possess are:
  \begin{enumerate}[(P1)]
  \item
    \label{item:WAlg-prop-1}
    $\HH_i$ is $i$-uniform and $V(\HH_i) = V(\HH)$,
  \item
    \label{item:WAlg-prop-2}
    $I$ is an independent set in $\HH_i$,
  \item
    \label{item:WAlg-prop-3}
    $\Delta_\ell(\HH_i) \le \Delta_\ell^i$ for each $\ell \in [i]$,
  \item
    \label{item:WAlg-prop-4}
    $e(\HH_i) \ge c_i p^{k-i} e(\HH)$.
  \end{enumerate}
\end{Prop}

Set $\HH_k = \HH$ and note that (P\ref{item:WAlg-prop-1})--(P\ref{item:WAlg-prop-4}) are vacuously satisfied for $i = k$. The main step of the Scythe Algorithm will be a procedure that, given $\HH_{i+1}$ and $I$ satisfying (P\ref{item:WAlg-prop-1})--(P\ref{item:WAlg-prop-4}), outputs a set $B_i \subseteq I$ of cardinality at most $b$, a set $A_i \subseteq V(\HH)$ with the property that $I \setminus B_i \subseteq A_i$, and a hypergraph $\HH_i$ satisfying (P\ref{item:WAlg-prop-1})--(P\ref{item:WAlg-prop-3}). Moreover, if the constructed $\HH_i$ does not satisfy (P\ref{item:WAlg-prop-4}), then we have $|A_i| \le (1-c_i)v(\HH)$. Crucially, these $A_i$ and $\HH_i$ depend solely on $B_i$ and $\HH_{i+1}$, that is, if on two inputs $(\HH_{i+1}, I)$ and $(\HH_{i+1}, I')$, the procedure outputs the same set $B_i$, it also outputs the same $A_i$ and $\HH_i$.

\begin{WAlg}
  Given an $(i+1)$-uniform hypergraph $\HH_{i+1}$ and an independent set $I \in \I(\HH_{i+1})$, set $\A_{i+1}^{(0)} = \HH_{i+1}$ and let $\HH_i^{(0)}$ be the empty hypergraph on the vertex set $V(\HH)$. For $j = 0, \ldots, b-1$, do the following:
  \begin{enumerate}[(1)]
  \item
    \label{item:WAlg-0}
    If $I \cap V\big( \A_{i+1}^{(j)} \big) = \emptyset$, then set $\HH_i = \HH_i^{(0)}$, $A_i = \emptyset$, and $B_i = \{u_0, \ldots, u_{j-1}\}$ and STOP.
  \item
    \label{item:WAlg-1}
    Let $u_j$ be the first vertex of $I$ in the max-degree order on $V\big( \A_{i+1}^{(j)} \big)$.
  \item \label{item:WAlg-2}
    Let $\HH_{i}^{(j+1)}$ be the hypergraph on the vertex set $V(\HH)$ defined by:
    \[
    \HH_{i}^{(j+1)} = \HH_{i}^{(j)} \cup \left\{D \in \binom{V(\HH)}{i} \colon D \cup \{u_j\} \in \A_{i+1}^{(j)} \right\}.
    \]
  \item
    \label{item:WAlg-3}
    Let $\A_{i+1}^{(j+1)}$ be the hypergraph on the vertex set $V\big( \A_{i+1}^{(j)} \big) \setminus W(u_j)$ defined by:\footnote{We emphasize that $W(u_j)$ is defined relative to the max-degree order on $V(\A_{i+1}^{(j)})$.}
    \[
    \A_{i+1}^{(j+1)} = \left\{ D \in \A_{i+1}^{(j)} \colon D \cap W(u_j) = \emptyset \text{ and } T \nsubseteq D \text{ for every } T  \in \bigcup_{\ell=1}^{i} M_\ell^i \big( \HH_{i}^{(j+1)} \big) \right\}.
    \]
  \end{enumerate}
  Finally, set $\HH_{i} = \HH_{i}^{(b)}$, $A_i = V\big( \A_{i+1}^{(b)} \big)$, and $B_i = \{u_0,\ldots,u_{b-1}\}$. 
\end{WAlg}

We shall now establish various properties of the Scythe Algorithm. We begin by making some basic (but key) observations.

\begin{lemma}
  \label{lemma:WAlg-basics}
  The following hold for every $i \in [k-1]$:
  \begin{enumerate}[(a)]
  \item
    \label{item:WAlg-basics-0}
    $\HH_i$ is $i$-uniform and $V(\HH_i) = V(\HH)$.
  \item
    \label{item:WAlg-basics-1}
    If $I \in \I(\HH_{i+1})$, then $I \in \I(\HH_i)$.
  \item
    \label{item:WAlg-basics-2}
    $B_i \subseteq I \subseteq A_i \cup B_i$.
  \item
    \label{item:WAlg-basics-3}
    The hypergraph $\HH_i$ and the set $A_i$ depend only on $\HH_{i+1}$ and the set $B_i$.
  \end{enumerate}
\end{lemma}

\begin{proof}
  Property (\ref{item:WAlg-basics-0}) is trivial. To see (\ref{item:WAlg-basics-1}), simply observe that each edge of $\HH_i$ is of the form $D \setminus \{u\}$ for some $D \in \HH_{i+1}$ and $u \in I$. Thus, if $I$ contains an edge of $\HH_i$, it must also contain an edge of $\HH_{i+1}$. To see (\ref{item:WAlg-basics-2}), observe that for each $j$, $u_j$ is the first vertex of $I$ in the max-degree order on $V\big( \A_{i+1}^{(j)} \big)$ and hence $W(u_j) \cap I = \{u_j\}$. It follows that $B_i \subseteq I$ and that $I \setminus A_i = B_i$. Note in particular that if $A_i = \emptyset$, then $I \cap V\big( \A_{i+1}^{(j)}\big) = \emptyset$ for some $j \in \{0, \ldots, b\}$, which implies that $B_i = I$. Finally, to prove (\ref{item:WAlg-basics-3}), observe that all steps of the Scythe Algorithm are deterministic and that every element of $I$ that we need to observe in order to define $A_i$ and $\HH_i$ is placed in $B_i$. More precisely, note that while choosing the vertex $u_j$, we only need to know the first vertex of $I$ in the max-degree order on $V\big(\A_{i+1}^{(j)}\big)$; the remaining vertices remain unobserved. Since we have $W(u_j) \cap B_i = W(u_j) \cap I = \{u_j\}$, this information can be recovered from $B_i$. Thus, at each step, the hypergraph $\HH_i^{(j+1)}$ can be recovered from $\HH_i^{(j)}$ and $B_i$, and the hypergraph $\A_{i+1}^{(j+1)}$ can be recovered from $\A_{i+1}^{(j)}$, $\HH_i^{(j+1)}$ and $B_i$. Hence, a trivial inductive argument proves that, if the algorithm does not stop in step~(\ref{item:WAlg-0}), for each $j \in \{0, \ldots, b\}$, the hypergraphs $\HH_i^{(j)}$ and $\A_{i+1}^{(j)}$ are determined by $\HH_{i+1}$ and the set $B_i$, as required.  Finally, the algorithm stops in step~(\ref{item:WAlg-0}) if and only if $|B_i| < b$. If this happens, then $\HH_i$ and $A_i$ are empty.
\end{proof}

We next show that the Scythe Algorithm exhibits a certain `consistency' while generating its output. This property will be important in the proof of Proposition~\ref{prop:main}.

\begin{lemma}
  \label{lemma:WAlg-consistency}
  Suppose that on inputs $(\HH_{i+1}, I)$ and $(\HH_{i+1}, I')$, the Scythe Algorithm outputs $(A_i, B_i, \HH_i)$ and $(A_i', B_i', \HH_i')$, respectively. If $B_i \subseteq I'$ and $B_i' \subseteq I$, then $(A_i, B_i, \HH_i) = (A_i', B_i', \HH_i')$.
\end{lemma}

\begin{proof}
  By Lemma~\ref{lemma:WAlg-basics}, it suffices to show that $B_i = B_i'$. Let us first consider the (degenerate) case when $\min\{|B_i|, |B_i'|\} < b$. Without loss of generality, we may assume that $|B_i| < b$. This means that, while running on $(\HH_{i+1}, I)$, the Scythe Algorithm stopped in step~(\ref{item:WAlg-0}). By Lemma~\ref{lemma:WAlg-basics}, it follows that $B_i = I$ and hence $B_i' \subseteq B_i$, which means that $|B_i'| < b$ and therefore $B_i' = I'$. Hence, $B_i = B_i'$, as claimed. On the other hand, if $|B_i| = |B_i'| = b$ and $B_i \neq B_i'$, then there must exist some $j$ such that $u_j \neq u_j'$. Let $j$ be the smallest such index. Note that by the minimality of $j$, we have $\A_{i+1}^{(j)} = \big(\A_{i+1}^{(j)}\big)' = \A$. Since $u_j \neq u_j'$, one of these vertices comes earlier in the max-degree order on $V(\A)$; without loss of generality, we may suppose that it is $u_j$. Since $B_i \subseteq I'$, it follows that $u_j \in I'$ and hence the Algorithm, while running on the input $(\HH_{i+1}, I')$, would not pick $u_j'$ in step $j$, a contradiction. This shows that in fact $B_i = B_i'$, as required.
\end{proof}

The next lemma shows that if $\HH_{i+1}$ satisfies (P\ref{item:WAlg-prop-3}), then so does $\HH_i$. The lemma follows easily from the definitions of $\Delta_\ell^i$ and $M_\ell^i(\G)$. 

\begin{lemma}
  \label{lemma:Delta}
  If $\Delta_{\ell+1}(\HH_{i+1}) \le \Delta_{\ell+1}^{i+1}$ for some $\ell \in [i]$, then $\Delta_\ell(\HH_{i}) \le \Delta_\ell^i$.
\end{lemma}

\begin{proof}
  The crucial observation is that if
  \[
  \deg_{\HH_i^{(j)}}(T) \ge \frac{\Delta_\ell^i}{2}
  \]
  for some $T \in \binom{V(\HH)}{\ell}$ and $j \in [b]$, then all edges containing $T$ are removed from $\A_{i+1}^{(j)}$ and hence no more such edges are added to $\HH_i$. It follows that $\deg_{\HH_i}(T) = \deg_{\HH_i^{(j)}}(T)$. Moreover, when we extend $\HH_{i}^{(j-1)}$ to $\HH_{i}^{(j)}$, then we only add to it sets $D$ such that $D \cup \{u_j\} \in \A_{i+1}^{(j-1)} \subseteq \HH_{i+1}$ and hence 
  \[
  \deg_{\HH_i^{(j)}}(T) - \deg_{\HH_i^{(j-1)}}(T) \le \deg_{\HH_{i+1}}(T \cup \{u_j\}) \le \Delta_{|T|+1}(\HH_{i+1}).
  \]
  It follows that
  \[
  \Delta_\ell(\HH_i) \le \frac{\Delta_\ell^i}{2} + \Delta_{\ell+1}(\HH_{i+1}) \le \frac{\Delta_\ell^i}{2} + \Delta_{\ell+1}^{i+1} \le \Delta_\ell^i
  \]
  where the last inequality follows from~\eqref{def:delta}.
\end{proof}

Next, let us establish an easy bound on the numbers $\Delta_1^i$.

\begin{lemma}
  \label{lemma:Delta-facts}
    $\Delta_1^i \le c 2^k p^{k-i} \frac{e(\HH)}{v(\HH)}$ for every $i \in \{1, \ldots, k\}$.
    \end{lemma}

\begin{proof}
  To prove the lemma, simply note that, by the definition of $\Delta_\ell^i$, for every $i \in [k]$ and every $\ell \in [i]$,
  \begin{equation}
    \label{eq:Deltali}
    \Delta_\ell^i =  2^dp^{k-i-d}\Delta_{d+\ell}(\HH) \quad \text{for some $d \in \{0, \ldots, k-i\}$}.
  \end{equation}
  One easily proves~\eqref{eq:Deltali} by induction on $k-i$. Intuitively, $d$ in~\eqref{eq:Deltali} is the number of times that the first term in the maximum in~\eqref{def:delta} is larger than the second term when following the recursive definition of $\Delta_\ell^i$ back to $\Delta_{d+\ell}^k$.
 
  Since $\Delta_\ell(\HH) \le c \cdot p^{\ell - 1} \frac{e(\HH)}{v(\HH)}$, as in the statement of Proposition~\ref{prop:main}, it follows from~\eqref{eq:Deltali} that
  \[
  \Delta_1^i \le \max_{0 \le d \le k - i}\left\{ 2^d p^{k-i-d} \Delta_{d+1}(\HH)  \right\} \le \max_{0 \le d \le k - i}\left\{ 2^d p^{k-i-d} \cdot c p^{d} \cdot \frac{e(\HH)}{v(\HH)} \right\} \le c \cdot 2^k p^{k-i}\frac{e(\HH)}{v(\HH)},
  \]
  as required.
\end{proof}

Finally, we show that if $\HH_{i+1}$ satisfies (P\ref{item:WAlg-prop-3}) and (P\ref{item:WAlg-prop-4}), then either $\HH_{i+1}$ also satisfies (P\ref{item:WAlg-prop-4}) or we have $|A_i| \le (1-c_i)v(\HH)$. Recall that $c_i = (ck2^{k+1})^{i-k}$.

\begin{lemma}
  \label{lemma:P4Ai}
  Let $i \in [k-1]$ and suppose that $e(\HH_{i+1}) \ge c_{i+1} p^{k - (i+1)} e(\HH)$ and that $\Delta_\ell(\HH_{i+1}) \le \Delta_\ell^{i+1}$ for every $\ell \in [i+1]$. Then either 
  \begin{equation}
    \label{eq:eHi-large}
    e(\HH_i) \ge \frac{p}{c \cdot 2^{k+1} k} e(\HH_{i+1}) \ge c_i p^{k-i} e(\HH)
  \end{equation}
  or $|A_i| \le (1 - c_i)v(\HH)$.
\end{lemma}

\begin{proof}
  If the Scythe Algorithm stops in step~(\ref{item:WAlg-0}), then $|A_i| = 0$ and there is nothing to prove. Hence, we may assume that steps~(\ref{item:WAlg-1})--(\ref{item:WAlg-3}) are executed $b$ times. Note that, for each $j \in \{0, \ldots, b-1\}$, we have
  \begin{equation}
    \label{eq:eHi-change}
    e\big( \HH_i^{(j+1)} \big) - e\big(\HH_i^{(j)} \big) = \deg_{\A_{i+1}^{(j)}}(u_j).    
  \end{equation}

  By the definition of the max-deg order, the right-hand side of~\eqref{eq:eHi-change} is at least the average degree of the hypergraph $\tilde{\A}_{i+1}^{(j)}$, the subhypergraph of $\A_{i+1}^{(j)}$ induced by the set $\big(V\big(\A_{i+1}^{(j)}\big) \setminus W(u_j)\big) \cup \{u_j\}$. Therefore, by the definition of $\A_{i+1}^{(j+1)}$, we have
  \[
  e\big( \HH_i^{(j+1)} \big) - e\big(\HH_i^{(j)} \big) \ge \frac{(i+1)e\big( \tilde{\A}_{i+1}^{(j)} \big)}{v\big( \tilde{\A}_{i+1}^{(j)} \big)} \ge  \frac{(i+1)e\big( \A_{i+1}^{(j+1)} \big)}{v\big( \HH \big)}.
  \]
  Hence, if $(i+1)e\big( \A_{i+1}^{(j+1)} \big) \ge e\big( \HH_{i+1} \big)$ for every $j \in \{0, \ldots, b-1\}$, then
  \[
  e(\HH_i) \ge \sum_{j = 0}^{b-1}  \frac{(i+1)e\big( \A_{i+1}^{(j+1)} \big)}{v\big( \HH \big)} \ge b \cdot \frac{e(\HH_{i+1})}{v(\HH)} = p \cdot e(\HH_{i+1}),
  \]
 since $b = p \cdot v(\HH)$, as required. Thus, we may assume that for some $j$,
  \begin{equation}
    \label{eq:eAb-small}
    e\big( \A_{i+1}^{(b)} \big) \le e\big( \A_{i+1}^{(j+1)} \big) < \frac{e\big( \HH_{i+1} \big)}{i+1}.    
  \end{equation}
  Intuitively, \eqref{eq:eAb-small} means that while running the Scythe Algorithm on $\HH_{i+1}$ and $I$, many edges are removed from $\A_{i+1}$ (that is, $\HH_{i+1}$) in step~(\ref{item:WAlg-3}). This may happen for one of the following two reasons: either many of the initial segments $W(u_j)$ are long or one of the families $M_\ell^i(\HH_i)$ of sets with high degree in $\HH_i$ is large.

  \begin{claim*}
    Either
    \[
    \sum_{j=0}^{b-1} | W(u_j) | \ge \frac{1}{4\Delta_1^{i+1}} \cdot e(\HH_{i+1})
    \]
    or for some $\ell \in [i]$,
    \[
    \left| M_\ell^i \big( \HH_i \big) \right| \ge \frac{1}{2(i+1)\Delta_\ell^{i+1}} \cdot e(\HH_{i+1}).
    \]
  \end{claim*}
  
  \begin{proof}[Proof of claim]
    Recall that $\A_{i+1}^{(0)} = \HH_{i+1}$ and observe that for every $j \in \{0, \ldots, b-1\}$,
    \begin{equation}
      \label{eq:eAi-change}
      e\big( \A_{i+1}^{(j)} \big) - e\big( \A_{i+1}^{(j+1)} \big) \le | W(u_j) | \cdot \Delta_1(\HH_{i+1})  + \sum_{\ell=1}^i \left| M_\ell^i \big( \HH_{i}^{(j+1)} \big) \setminus M_\ell^i \big( \HH_{i}^{(j)} \big) \right| \cdot \Delta_\ell(\HH_{i+1}).
    \end{equation}
    Inequality~\eqref{eq:eAi-change} follows since in step (\ref{item:WAlg-3}) of the Scythe Algorithm, we remove from $\A_{i+1}^{(j)}$ only the edges that contain either a vertex of $W(u_j)$ or a member of $M_\ell^i \big( \HH_{i}^{(j+1)} \big)$ for some $\ell \in [i]$. Thus, since $\Delta_\ell(\HH_{i+1}) \le \Delta_\ell^{i+1}$ for every $\ell \in [i]$, summing~\eqref{eq:eAi-change} over all $j$, we get
    \[
    e\big( \HH_{i+1} \big) - e(\A_{i+1}^{(b)}) \le \sum_{j=0}^{b-1} |W(u_j)| \cdot \Delta_1^{i+1} + \sum_{\ell = 1}^i \left| M_\ell^i \big( \HH_{i}^{(b)} \big) \right| \cdot  \Delta_\ell^{i+1}.
    \]
    Since we assumed that $e(\A_{i+1}^{(b)}) < e\big( \HH_{i+1} \big) / (i+1)$, see~\eqref{eq:eAb-small}, and $\HH_i = \HH_i^{(b)}$, it follows that if 
    \[
    \sum_{j=0}^{b-1} \big| W(u_j) \big| \cdot \Delta_1^{i+1} < \frac{e(\HH_{i+1})}{4} \le \frac{i}{2(i+1)} \cdot e(\HH_{i+1}),
    \]
    then
    \[
    \left| M_\ell^i \big( \HH_i \big) \right| \cdot \Delta_\ell^{i+1} \ge \frac{1}{2(i+1)} \cdot e(\HH_{i+1}) \quad \text{for some $\ell \in [i]$},
    \]
    as claimed.
  \end{proof}
  
  Finally, let us deal with the two cases implied by the claim. In the remainder of the proof, we will show that if $M_\ell^i \big( \HH_i \big)$ is large for some $\ell \in [i]$, then $e(\HH_i)$ is large and if $\sum_{j=0}^{b-1} | W(u_j) |$ is large, then $|A_i|$ is small.
  
  \medskip
  \noindent
  \textbf{Case 1:} $\left| M_\ell^i \big( \HH_i \big) \right| \ge \frac{1}{2(i+1)\Delta_\ell^{i+1}} \cdot e(\HH_{i+1})$ for some $\ell \in [i]$.
  \smallskip

Since $\deg_{\HH_i}(T) \ge \Delta_\ell^i/2$ for every $T \in M_\ell^i \big( \HH_i \big)$, it follows by the handshaking lemma that
  \begin{equation}
    \label{eq:edges2}
    e(\HH_i) = \binom{i}{\ell}^{-1} \sum_{T \in \binom{V(\HH)}{\ell}} \deg_{\HH_{i}}(T) \ge \frac{\big| M_\ell^i(\HH_i) \big| \cdot \Delta_\ell^i}{2\binom{i}{\ell}}.
  \end{equation}
  Recalling that $\Delta_\ell^i \ge p \Delta_{\ell}^{i+1}$, see~\eqref{def:delta}, we have
  \[
  e(\HH_{i}) \ge \frac{e(\HH_{i+1})}{4(i+1) \binom{i}{\ell}} \cdot  \frac{\Delta_\ell^i}{\Delta_\ell^{i+1}} \ge \frac{p}{2^{i+2}(i+1)} \cdot e(\HH_{i+1}) \ge \frac{p}{2^{k+1} k } \cdot  e(\HH_{i+1}),
  \]
  as required.
  
  \medskip
  \noindent
  \textbf{Case 2:} $\sum_{j=0}^{b-1} |W(u_j)| \ge \frac{1}{4\Delta_1^{i+1} } \cdot e(\HH_{i+1})$.
  \medskip

  We claim that in this case, $|A_i| \le (1 - c_i) v(\HH)$. Indeed, we have
  \[
  v(\HH) - |A_i| = v\big( \A_{i+1}^{(0)} \big) - v\big( \A_{i+1}^{(b)} \big) = \sum_{j=0}^{b-1} |W(u_j)| \ge \frac{e(\HH_{i+1})}{4\Delta_1^{i+1}}.
  \]
  Recall that $\Delta_1^{i+1} \le c 2^k p^{k-i-1} \frac{e(\HH)}{v(\HH)}$ by Lemma~\ref{lemma:Delta-facts}. Thus,
  \[
  v(\HH) - |A_i|  \ge \frac{p^{i+1-k}}{c 2^{k+2}} \cdot \frac{v(\HH)}{e(\HH)} \cdot e(\HH_{i+1}) \ge c_i v(\HH),
  \]
  since $e(\HH_{i+1}) \ge c_{i+1} p^{k - (i+1)} e(\HH)$ and $c_{i+1} / (c2^{k+2}) \ge c_i$.
\end{proof}

\subsection{The proof of Proposition~\ref{prop:main} and Theorem~\ref{thm:main}}

\begin{proof}[Proof of Proposition~\ref{prop:main}] 
  Let $k$ be an integer and let $c$ be a positive constant. Furthermore, let $p \in (0,1)$ and let $\HH$ be a $k$-uniform hypergraph that satisfy the assumptions of Proposition~\ref{prop:main}. Let $\delta = (ck2^{k+1})^{-k}$ and $b = pv(\HH)$. We will use the Scythe Algorithm, described in Section~\ref{sec:WSAlg}, to construct a family $\S$ and functions $f_0$ and $g_0$ as in the statement of Proposition~\ref{prop:main}. We  obtain them by running the following algorithm (with $\HH_k = \HH$) on every independent set $I \in \I(\HH)$. We shall define $f_0$ somewhat implicitly by defining a function $f_0^* \colon \I(\HH) \to \cP(V(\HH))$ that is constant on the set $g_0^{-1}(S)$ for every $S \in \S$.

  \begin{Proc}
    Given an $I \in \I(\HH)$, set $i = k - 1$ and repeat the following:
    \begin{enumerate}[(1)]
    \item
      Apply the Scythe Algorithm to $\HH_{i+1}$ and $I$. Suppose that it outputs $\HH_i$, $A_i$ and $B_i$.
    \item
      \label{item:Alg-2}
      If $|A_i| \le (1 - \delta) v(\HH)$, then set $q = i$, $r = i+1$ and STOP.
    \item
      If $i > 1$, then set $i = i - 1$. Otherwise, set $q = r = 1$ and STOP.
    \end{enumerate}
  \end{Proc}

  Let $I$ be an independent set and let us execute the above procedure (with $\HH_k = \HH$) on $I$. We claim that for every $i \in \{r, \ldots, k\}$, the hypergraph $\HH_i$ satisfies properties (P\ref{item:WAlg-prop-1})--(P\ref{item:WAlg-prop-4}) defined in Section~\ref{sec:WSAlg}. This follows by induction on $k - i$. The base of the induction, the case $i = k$, follows vacuously from the definitions of $c_k$ and $\Delta_\ell^k$ for $\ell \in [k]$. The inductive step follows from Lemmas~\ref{lemma:WAlg-basics}, \ref{lemma:Delta}, and~\ref{lemma:P4Ai}. To see this, note that since $|A_i| > (1-\delta)v(\HH) \ge (1-c_i)v(\HH)$ for all $i \in \{r, \ldots, k-1\}$, then~\eqref{eq:eHi-large} in Lemma~\ref{lemma:P4Ai} always holds.

  Now, let us define $g_0(I)$ and $f_0^*(I)$. Suppose first that $r > 1$ and note that in this case, the algorithm stopped in step (\ref{item:Alg-2}), which means that $|A_q| \le (1 - \delta) v(\HH)$; we set 
  \[
  g_0(I) = B_{k-1} \cup \ldots \cup B_q \quad \text{and} \quad f_0^*(I) = A_q.
  \]
  On the other hand, if $r = 1$, then we set
  \[
  g_0(I) = B_{k-1} \cup \ldots \cup B_1 \quad \text{and} \quad f_0^*(I) = \big\{ v \in V(\HH_1) \colon \{v\} \not\in \HH_1 \big\}.
  \]
  Finally, we let
  \[
  \S  = \{ g_0(I) \colon I \in \I(\HH) \}.
  \]
  We will define $f_0$ by letting $f_0(S) = f_0^*(I)$ for some $I \in g_0^{-1}(S)$. We first show that this definition will not depend on the choice of $I$. In fact, we shall prove a slightly stronger statement, which also establishes the consistency property of $g_0$ stated in the final line of Proposition~\ref{prop:main}.

  \begin{claim*}
    Suppose that for some $I, I' \in \I(\HH)$, $g_0(I) \subseteq I'$ and $g_0(I') \subseteq I$. Then $g_0(I) = g_0(I')$ and $f_0^*(I) = f_0^*(I')$.
  \end{claim*}

  \begin{proof}[Proof of claim]
    Suppose that while running the algorithm on some $I$, we obtain a sequence $(B_{k-1}, \ldots, B_q)$. Since $g_0(I)$ depends solely on $(B_{k-1}, \ldots, B_q)$ and, by Lemma~\ref{lemma:WAlg-basics}, for each $i$, the hypergraph $\HH_i$ and the set $A_i$ depend only on $(B_{k-1}, \ldots, B_i)$, then also $f_0^*(I)$ depends solely on $(B_{k-1}, \ldots, B_q)$. Hence, it suffices to show that if, while running the algorithm on some $I'$ with $B_{k-1} \cup \ldots \cup B_q \subseteq I'$, we obtain a sequence $(B_{k-1}', \ldots, B_{q'}')$ with $B_{k-1}' \cup \ldots \cup B_{q'}' \subseteq I$, then $(B_{k-1}', \ldots, B_{q'}') = (B_{k-1}, \ldots, B_q)$. To this end, let us first observe that, under the above assumptions, for every $i \in [k-1]$, if $\HH_{i+1} = \HH_{i+1}'$, then $B_i = B_i'$. Indeed, note that $B_i$ and $B_i'$ are the outputs of the Scythe Algorithm executed on the inputs $(\HH_{i+1}, I)$ and $(\HH_{i+1}', I')$, respectively. Hence, if $\HH_{i+1} = \HH_{i+1}'$, then since
    \[
    B_i \subseteq B_{k-1} \cup \ldots \cup B_q \subseteq I' \quad \text{and} \quad B_i' \subseteq B_{k-1}' \cup \ldots \cup B_{q'}' \subseteq I,
    \]
    then Lemma~\ref{lemma:WAlg-consistency} implies that $B_i = B_i'$. Since clearly $\HH_k = \HH_k' = \HH$ and, as noted before, for each $i$, $\HH_{i+1}$ depends only on $(B_{k-1}, \ldots, B_{i+1})$, it follows that $B_i = B'_i$ for all $i$, as required.
  \end{proof}

  By the above claim, we can define $f_0$ by letting, for every $S \in \S$, $f(S) = f_0^*(I)$ for any $I \in g_0^{-1}(S)$. Finally, let us show that the $\S$, $g_0$, and $f_0$, which we have just defined, satisfy the required conditions, that is, for all $I, I' \in \I(\HH)$,
  \begin{enumerate}[(i)]
  \item
    \label{item:cond-S}
    $|S| \le (k-1)pv(\HH)$ for every $S \in \S$,
  \item
    \label{item:cond-gfI}
    $g_0(I) \subseteq I \subseteq f_0(g_0(I)) \cup g_0(I)$,
  \item
    \label{item:cond-f-size}
    $|f_0(g_0(I))| \le (1-\delta)v(\HH)$,
  \item
    \label{item:cond-g-consistency}
    $g_0(I) \subseteq I'$ and $g_0(I') \subseteq I$ imply that $g_0(I) = g_0(I')$.
  \end{enumerate}

  To see (\ref{item:cond-S}), simply recall that $|B_i| \le p v(\HH)$ for every $i \in [k-1]$. To see (\ref{item:cond-gfI}), note that $B_i \subseteq I \subseteq A_i \cup B_i$  for every $i \in \{q, \ldots, k-1\}$, by Lemma~\ref{lemma:WAlg-basics}, that $I$ is an independent set in $\HH_1$ (if $r = 1$) and, crucially, that $f_0(g_0(I)) = f_0^*(I)$. To see (\ref{item:cond-f-size}), note that if $r > 1$, then $|A_q| \le (1 - \delta)v(\HH)$, see step (\ref{item:Alg-2}) of the algorithm; if $r = 1$, then since $\Delta_1(\HH_1) \le c \cdot 2^k p^{k-1} e(\HH) / v(\HH)$, by Lemma~\ref{lemma:Delta-facts} and property~(P\ref{item:WAlg-prop-3}), we have
$$\left| \big\{ v \in V(\HH_1) \colon \{v\} \in \HH_1 \big\} \right| \ge \frac{e(\HH_1)}{\Delta_1(\HH_1)} \ge \frac{c_1p^{k-1}e(\HH)}{\Delta_1(\HH_1)} \ge \delta v(\HH),$$
since $\delta \le c_1 / ( c \cdot 2^k)$ and $\HH_1$ satisfies property (P\ref{item:WAlg-prop-4}), so $e(\HH_1) \ge c_1p^{k-1}e(\HH)$. Finally, (\ref{item:cond-g-consistency}) follows directly from the claim.
\end{proof}

\begin{proof}[Proof of Theorem~\ref{thm:main}] 
  The theorem follows by applying Proposition~\ref{prop:main} a bounded number of times. Given an integer $k$ and positive reals $c$ and $\eps$, let $\delta = \delta_{\ref{prop:main}}(c/\eps)$ and let
  \[
  C = (k-1) \cdot \left(\frac{1}{\delta} \log \frac{1}{\eps} + 1\right).
  \]
  Let $V$ be a finite set and let $\F$ be an increasing family of subsets of $V$ such that $|A| \ge \eps |V|$ for every $A \in \F$. Let $p \in (0,1)$ and suppose that $\HH$ is a $k$-uniform hypergraph on the vertex set $V$ that is $(\F, \eps)$-dense and satisfies the assumptions of the theorem, that is, 
    \[
  \Delta_\ell(\HH) \le c p^{\ell-1}\frac{e(\HH)}{v(\HH)}
  \]
  for every $\ell \in [k]$. We now show how to construct a family $\S \subseteq \binom{V(\HH)}{\le C p v(\HH)}$ and functions $f \colon \S \to \ol{\F}$ and $g \colon \I(\HH) \to \S$ such that
  \begin{equation}
    \label{eq:fg-cond}
    g(I) \subseteq I \quad \text{and} \quad I \setminus g(I) \subseteq f(g(I))
  \end{equation}
  for every $I \in \I(\HH)$. Similarly as in the proof of Proposition~\ref{prop:main}, we shall define $f$ via a function $f^* \colon \I(\HH) \to \cP(V)$ that is constant on each set $g^{-1}(S)$ with $S \in \S$.
  
  Fix some $I \in \I(\HH)$. Using Proposition~\ref{prop:main}, we shall construct (for some $J \le \frac{1}{\delta} \log \frac{1}{\eps} + 1$) a sequence $(A_j, S_j)_{j=1}^J$ of pairs of subsets of $V$ such that for each $j \in [J]$,
  \[
  S_1 \cup \ldots \cup S_j \subseteq I \subseteq A_j \cup S_1 \cup \ldots \cup S_j.
  \]
  Moreover, $A_J \in \ol{\F}$ while $|S_1 \cup \ldots \cup S_J| \le Cpv(\HH)$. Crucially, the set $A_J$ will depend solely on $S_1 \cup \ldots \cup S_J$. We will let $g(I) = S_1 \cup \ldots \cup S_J$ and $f^*(I) = A_J$.

  \begin{Const}
    Let $S_0 = \emptyset$ and let $A_0 = V$. For $j = 0, 1, \ldots$, do the following:
    \begin{enumerate}[(1)]
    \item
      \label{item:Const-1}
      If $A_j \in \F$, then let $I_j = I \cap A_j$ and apply Proposition~\ref{prop:main} with $c_{\ref{prop:main}} = c/\eps$ and $p_{\ref{prop:main}} = p$ to the hypergraph $\HH[A_j]$ and the set $I_j$ to obtain sets $g_0(I_j)$ and $f_0(g_0(I_j))$ such that $g_0(I_j) \subseteq I_j$ and $I_j \setminus g_0(I_j) \subseteq f_0(g_0(I_j))$. Otherwise, if $A_j \in \ol{\F}$, then STOP.
    \item
      Let $S_{j+1} = g_0(I_j)$ and let $A_{j+1} = f_0(g_0(I_j))$.
    \end{enumerate}
  \end{Const}
  Let us first show that the above procedure is well-defined, that is, that the assumptions of Proposition~\ref{prop:main} are satisfied each time we are in~(\ref{item:Const-1}). To this end, fix some $A \subseteq V$ and note that if $A \in \F$, then, since $\HH$ is $(\F, \eps)$-dense,
  \[
  \Delta_\ell(\HH[A]) \le \Delta_\ell(\HH) \le c p^{\ell-1}\frac{e(\HH)}{v(\HH)} \le \frac{c}{\eps} \cdot p^{\ell-1}\frac{e(\HH[A])}{v(\HH[A])},
  \]
where the last step follows since $e(\HH[A]) \ge \eps \cdot e(\HH)$ and $v(\HH[A]) \le v(\HH)$.
  
  Next, let us show that the above procedure terminates, therefore producing a finite sequence $(A_j, S_j)$ with $j \in [J]$. To this end, let us simply note that by Proposition~\ref{prop:main}, $|A_{j+1}| \le (1-\delta)|A_j|$ for all $j$, $A_0 = V$ and $|A| \ge \eps|V|$ for every $A \in \F$. Moreover, since $A_{J-1} \in \F$, then
  \[
  \eps |V| \le |A_{J-1}| \le (1-\delta)^{J-1} |A_0| = \exp(-(J-1)\delta)|V|
  \]
  and hence $J \le \frac{1}{\delta} \log \frac{1}{\eps} + 1$. It immediately follows that
  \[
  |g(I)| \le \sum_{j=1}^{J} |S_j| \le \sum_{j=1}^J (k-1)pv(\HH[A_j]) \le J(k-1)pv(\HH) \le Cpv(\HH).
  \]
  Finally, let $\S = \{g(I) \colon I \in \I(\HH)\}$. It remains to show that for every $S \in \S$, $f^*$ is constant on $g^{-1}(S)$. Similarly as in the proof of Proposition~\ref{prop:main}, we shall prove a somewhat stronger statement.
  
  \begin{claim*}
    Suppose that for some $I, I' \in \I(\HH)$, $g(I) \subseteq I'$ and $g(I') \subseteq I$. Then $g(I) = g(I')$ and $f^*(I) = f^*(I')$.
  \end{claim*}

  \begin{proof}[Proof of claim]
    Suppose that while running the above procedure on some $I$, we generate a sequence $(A_j, S_j)_{j=1}^J$. Since for each $j$, $A_{j+1}$ depends solely on $A_j$ and $S_{j+1}$, where $A_0 = V$, then both $g(I)$ and $f^*(I)$ depend solely on $(S_1, \ldots, S_J)$. Hence, it suffices to show that if, while running the above procedure on some $I'$ with $S_1 \cup \ldots \cup S_J \subseteq I'$, we generate a sequence $(A_j', S_j')_{j=1}^{J'}$ with $S_1' \cup \ldots \cup S_{J'}' \subseteq I$, then $(S_1, \ldots, S_J) = (S_1', \ldots, S_{J'}')$. To this end, it suffices to note that if $A_j = A_j'$, then, since
    \[
    S_{j+1} \subseteq S_1 \cup \ldots \cup S_J \subseteq I' \quad \text{and} \quad  S_{j+1}' \subseteq S_1' \cup \ldots \cup S_{J'}' \subseteq I,
    \]
    by the consistency property of $g_0$ stated in the final line of Proposition~\ref{prop:main}, $S_{j+1} = S_{j+1}'$. Since $A_0 = A_0' = V$ and for each $j$, $A_j$ depends only on $(S_1, \ldots, S_j)$, it follows that $S_j = S_j'$ for all $j$, as required.
  \end{proof}
  Finally, for every $S \in \S$, we let $f(S) = f^*(I)$ for some $I \in g^{-1}(S)$. This completes the proof of Theorem~\ref{thm:main}.
\end{proof}

\section{Szemer\'edi's theorem for sparse sets}

\label{sec:Sz}

In this section, we prove Theorem~\ref{thm:Sz} and derive from it Corollary~\ref{cor:Sz}. Before we get to the proofs, let us first remark that Theorem~\ref{thm:Sz} and Corollary~\ref{cor:Sz} are both sharp up to the value of the constant $C$ in the lower bounds for $p$ and $m$. More precisely, let us make the following two observations.

\begin{enumerate}
\item
  For every $\beta \in (0,1)$, there is a positive $c$ such that if $m \le cn^{1-1/(k-1)}$, then the number of $m$-subsets of~$[n]$ that contain no $k$-term AP is at least $(1-\beta)^m\binom{n}{m}$. To see this, let $\eps = \beta^2$ and observe that if $c$ is sufficiently small and $m \le cn^{1-1/(k-1)}$, then the expected number of $k$-term APs in a random $(1+\eps)m$-subset of~$[n]$ is smaller than $\eps m/2$ and hence by Markov's inequality, at least half of all $(1+\eps)m$-subsets of~$[n]$ contain a subset of size $m$ with no $k$-term AP. Hence\footnote{We assume here, without loss of generality, that $\beta$ (and hence also $\eps$) is sufficiently small.}
  \[
  \text{\#\{$m$-subsets of $[n$] with no $k$-term AP\}} \ge \frac{\binom{n}{(1+\eps)m}}{2\binom{n}{\eps m}} \ge \left(1-\sqrt{\eps}\right)^m \binom{n}{m},
  \]
where the final inequality holds since $\binom{n}{(1+\eps)m} \ge \big( \frac{n}{2m} \big)^{\eps m} \binom{n}{m}$ and $\binom{n}{\eps m} \le \big( \frac{en}{\eps m} \big)^{\eps m}$.
\item 
   There is a positive constant $c$ such that if $p_n \le cn^{-1/(k-1)}$, then
   \[
   \Pr\big( \text{$[n]_{p_n}$ is $(\delta,k)$-Szemer{\'e}di} \big) \to 0 \text{ as $n \to \infty$}.
   \]
   For a (simple) proof of this statement, we refer the reader to~\cite{Sch}.
\end{enumerate}

We shall in fact prove the following somewhat stronger version of Corollary~\ref{cor:Sz}, originally proved by Schacht~\cite{Sch} (the approach of Conlon and Gowers~\cite{CG} yields a somewhat weaker probability estimate).

\begin{cor}
  \label{cor:Sz-strong}
  For every $k \in \N$ and every $\delta \in (0,1)$, there exists a constant $C$ such that for all sufficiently large $n$, if $p \ge Cn^{-1/(k-1)}$, then
  \[
  \Pr\big( \text{$[n]_p$ is $(\delta,k)$-Szemer{\'e}di} \big) \ge 1 - 2\exp(-pn/8).
  \]
\end{cor}

In the proofs of Theorem~\ref{thm:Sz} and Corollary~\ref{cor:Sz-strong}, and frequently in later sections, we shall need various estimates on binomial coefficients, which we list here for future reference. Let $a$, $b$, and $c$ be integers satisfying $a \ge b \ge c \ge 0$. Then the following inequalities hold:

\begin{minipage}{0.45\textwidth} 
  \begin{align}
    \label{eq:binomial-1}
    \binom{a}{b} & \le \bigg( \frac{ea}{b} \bigg)^b, \\
    \addtocounter{equation}{1}
    \label{eq:binomial-2}
    \binom{a}{b-c} & \le \left(\frac{b}{a-b}\right)^c\binom{a}{b},
  \end{align}
\end{minipage} 
\begin{minipage}{0.45\textwidth} 
  \begin{align}
    \addtocounter{equation}{-2}
    \label{eq:binomial-3}
    \binom{b}{c} & \le \left(\frac{b}{a}\right)^c\binom{a}{c}, \\
    \addtocounter{equation}{1}
    \label{eq:binomial-4}
    \binom{a}{c} & \le \left(\frac{a-c}{b-c}\right)^c\binom{b}{c}.
  \end{align}
\end{minipage}

\smallskip
\noindent
We remark that each inequality above follows easily from the definition of $\binom{a}{b}$.

\begin{proof}[Proof of Corollary~\ref{cor:Sz-strong}]
  Fix $k \in \N$ and $\delta \in (0,1)$, let $\beta = \delta/(2e) \cdot e^{-1/\delta}$, and set $C = 2C_{\ref{thm:Sz}}(\beta,k)/\delta$. Assume that $p \ge Cn^{-1/(k-1)}$, let $m = \delta pn/2$, and let $X_m$ denote the number of $m$-subsets of $[n]_p$ that contain no $k$-term AP. By Theorem~\ref{thm:Sz} and~\eqref{eq:binomial-1}, we have
  \begin{equation}
    \label{eq:ExXm}
    \Pr( X_m > 0 ) \le \Ex[X_m] \le \binom{\beta n}{m}p^m \le \left( \frac{\beta e p n}{m} \right)^m = \left( \frac{2\beta e}{\delta} \right)^m = e^{-m/\delta}.
  \end{equation}
  Let $\A$ denote the event that $[n]_p$ is \emph{not} $(\delta,k)$-Szemer{\'e}di, i.e., that $[n]_p$ contains a subset with $\delta |[n]_p|$ elements and no $k$-term AP. By~\eqref{eq:ExXm} and Chernoff's inequality (see, e.g., \cite[Appendix~A]{AlSp}), it follows that
  \[
  \Pr(\A) \le \Pr\left(\A \wedge |[n]_p| \ge \frac{pn}{2}\right) + \Pr\left( |[n]_p| < \frac{pn}{2} \right) \le \Pr(X_m > 0) + e^{-pn/8} \le 2e^{-pn/8},
  \]
  as required.
\end{proof}

Finally, let us show how to deduce Theorem~\ref{thm:Sz} from Theorem~\ref{thm:main}. Our proof will use the following robust version of Szemer{\'e}di's theorem, which can be proved by a simple averaging argument, originally observed by Varnavides~\cite{Varn}.

\begin{lemma}
  \label{lemma:Varn}
  For every positive $\delta$ and $k \in [n]$, there exists a positive $\eps$ such that the following holds for all sufficiently large $n$. Every subset of $[n]$ with at least $\delta n$ elements contains at least $\eps n^2$ $k$-term APs.
\end{lemma}

\begin{proof}[Proof of Theorem~\ref{thm:Sz}]
  Given $k \in \N$ and positive $\beta$, let $\delta = \min\{ \beta/2, 1/10 \}$ and let $n \in \N$ be sufficiently large. Let $\HH$ be the $k$-uniform hypergraph of $k$-term APs in $[n]$, i.e., the hypergraph on the vertex set $[n]$ whose edges are all $k$-term APs in $[n]$, let $\F$ denote the family of subsets of $[n]$ with at least $\delta n$ elements, and let $\eps = \eps_{\ref{lemma:Varn}}(\delta, k)$. By Lemma~\ref{lemma:Varn}, the hypergraph $\HH$ is $(\F, \eps)$-dense, provided that $n$ is sufficiently large. Let $p = n^{-1/(k-1)}$ and let $c = 2k^2$. Since $e(\HH) \ge n^2/k^2 \ge 2n^2/c$, it follows that 
  \[
  \Delta_1(\HH) \le k \cdot \frac{n}{k-1} \le 2n \le c \cdot p^{1-1} \frac{e(\HH)}{v(\HH)},
  \]
   for every $\ell \in \{2, \ldots, k-1\}$,
  \[
  \Delta_\ell(\HH) \le \Delta_2(\HH) \le \binom{k}{2} \le 2n^{1/(k-1)} \le c \cdot p^{\ell-1}\frac{e(\HH)}{v(\HH)},
  \]
 and $\Delta_k(\HH)  = 1 \le c \cdot p^{k-1}\frac{e(\HH)}{v(\HH)}$. 
 
  Let $C' = C_{\ref{thm:main}}(k,\eps,c)$, let $C = C'/\delta$, and assume that $m \ge Cn^{1-1/(k-1)} = Cpn$. Note that if $m > \delta n/2$, then $\I(\HH, m) = 0$ by Szemer{\'e}di's theorem, so we may assume that $m \le \delta n /2$. Since $C'pn \le \delta m$, then by Theorem~\ref{thm:main}, there exists a family $\S \subseteq \binom{[n]}{\le C'pn} \subseteq \binom{[n]}{\le \delta m}$ and functions $f \colon \S \to \ol{\F}$ and $g \colon \I(\HH) \to \S$, such that for every $I \in \I(\HH)$,
  \[
  g(I) \subseteq I \quad \text{and} \quad I \setminus g(I) \subseteq f(g(I)).
  \]
  Therefore, using \eqref{eq:binomial-1} and \eqref{eq:binomial-2}, the number of independent sets of size $m$ in $\HH$ can be estimated as follows:
  \begin{align*}
    \label{eq:Sz-IHm}
    |\I(\HH, m)| & \, = \, \sum_{S \in \S} |\{ I \in \I(\HH, m) \colon g(I) = S \}| \le \sum_{S \in \S} \binom{|f(S)|}{m-|S|} \\
    &\, \le \, \sum_{k \le \delta m} \binom{n}{k} \binom{\delta n}{m - k} \le \sum_{k \le \delta m}  \left( \frac{en}{k} \right)^k \left( \frac{m}{\delta n - m} \right)^k \binom{\delta n}{m}. 
  \end{align*}
  Since $m \le \delta n /2$ and the function $x \mapsto (y/x)^x$ is increasing on $(0,y/e)$, it follows that
  \[
  |\I(\HH,m)| \le \sum_{k \le \delta m} \left( \frac{2em}{\delta k} \right)^k \binom{\delta n}{m} \le m \left( \frac{2e}{\delta^2} \right)^{\delta m} \binom{\delta n}{m} \le \binom{\beta n}{m},
  \]
  where the final inequality follows since $\binom{\delta n}{m} \le 2^{-m} \binom{2\delta n}{m}$, by~\eqref{eq:binomial-3}, and since $2^{1/\delta} > 2e / \delta^2$ if $\delta \le 1/10$. This proves Theorem~\ref{thm:Sz}.
\end{proof}

The same proof, combined with an analogue of Lemma~\ref{lemma:Varn} due to Furstenberg and Katznelson~\cite{FK}, yields the following generalization of Theorem~\ref{thm:Sz}, which strengthens both~\cite[Theorem~10.4]{CG} and~\cite[Theorem~2.3]{Sch}. Given a set $F \subseteq \N^\ell$, we call a set  of the form $a + bF = \{a + bx \colon x \in F\}$, with $a \in \N^\ell$ and $b \in \ZZ \setminus \{0\}$, a \emph{homothetic copy} of $F$.

\begin{thm}
  For every positive $\beta$, every $\ell \in \N$, and every finite configuration $F \subseteq \N^\ell$, there exist constants $C$ and $n_0$ such that the following holds. For every $n \in \N$ with $n \ge n_0$, if $m \ge Cn^{\ell - 1/(|F| - 1)}$, then there are at most
  \[
  \binom{\beta n^\ell}{m}
  \]
  $m$-subsets of $[n]^\ell$ that contain no homothetic copy of $F$.
\end{thm}

Finally, using the famous polynomial Szemer\'edi theorem of Bergelson and Leibman~\cite{BL}, the same argument gives a counting version of~\cite[Theorem~10.7]{CG}. 

\begin{thm}
  For every positive $\beta$ and integers $k$ and $r$, there exist constants $C$ and $n_0$ such that the following holds. For every $n \in \N$ with $n \ge n_0$, if $m \ge C n^{1 - 1/kr}$, then there are at most
  \[
  \binom{\beta n}{m}
  \]
  $m$-subsets of $[n]$ that contain no set of the form $\{a, a + d^r, \ldots, a + k d^r \}$.
\end{thm}

\section{Extremal results for sparse sets}

\label{sec:extremal-results}

In this section, we shall deduce from Theorem~\ref{thm:main} two versions of the general transference theorem of Schacht~\cite[Theorem~3.3]{Sch}. We remind the reader that a statement very similar to Schacht's theorem was proved independently by Conlon and Gowers~\cite{CG}. For the benefit of the readers who are familiar with~\cite{Sch}, we shall state it using the terminology used there.

\begin{defn}
  \label{defn:alpha-dense}
  Let $\HH = (\HH_n)_{n \in \N}$  be a sequence of $k$-uniform hypergraphs and let $\alpha \in [0,1)$. We say that $\HH$ is \emph{$\alpha$-dense} if the following is true: For every positive $\delta$, there exist positive $\eps$ and $n_0$ such that for every $n$ with $n \ge n_0$ and every $U \subseteq V(\HH_n)$ with $|U| \ge (\alpha + \delta)v(\HH_n)$, we have
  \[
  e(\HH_n[U]) \ge \eps e(\HH_n).
  \]
\end{defn}

Let us remark here that Definition~\ref{defn:Feps-dense} is a generalization of Definition~\ref{defn:alpha-dense}. Indeed, if $\F_\delta$ denotes the collection of all subsets of $V(\HH_n)$ with at least $(\alpha + \delta)v(\HH_n)$ elements, then a sequence $\HH$ of hypergraphs is $\alpha$-dense if and only if for every positive $\delta$, there exists a positive $\eps$ such that for all sufficiently large $n$, the hypergraph $\HH_n$ is $(\F_\delta, \eps)$-dense.

We start with the `random' version of our extremal result, which was originally proved by Schacht~\cite[Theorem~3.3]{Sch}.

\begin{thm}
  \label{thm:Sch-weak}
  Let $\HH$ be a sequence of $k$-uniform hypergraphs, let $\alpha \in [0,1)$, and let $c$ be a  positive constant. Suppose that $\p \in [0,1]^\N$ is a sequence of probabilities such that for all sufficiently large $n \in \N$, and for every $\ell \in [k]$, we have
  \begin{equation}
    \label{eq:Delta-ell}
    \Delta_\ell(\HH_n) \le c \cdot p_n^{\ell-1} \frac{e(\HH_n)}{v(\HH_n)}.
  \end{equation}
  If $\HH$ is $\alpha$-dense, then the following holds. For every positive $\delta$, there exists a constant $C$ such that if $q_n \ge Cp_n$ and $q_nv(\HH_n) \to \infty$ as $n \to \infty$, then a.a.s.
  \[
  \alpha\big( \HH_n[V(\HH_n)_{q_n}] \big) \le (\alpha + \delta) q_n v(\HH_n).
  \]
\end{thm}

We note that the probability bounds implicit in the `asymptotically almost surely' statement that we obtain are, as in~\cite{Sch}, optimal, that is, they decay exponentially in $p_n v(\HH_n)$. 

\begin{remark}
  We remark that the only difference between Theorem~\ref{thm:Sch-weak} and \cite[Theorem~3.3]{Sch} are the assumptions on the hypergraph sequence $\HH$. It turns out that this difference is only superficial, since condition~\eqref{eq:Delta-ell} is essentially equivalent to the condition that $\HH$ is $(K,\p)$-bounded (see~\cite{Sch}). One easily checks that if $\HH_n$ satisfies~\eqref{eq:Delta-ell} for sufficiently large $n$, then $\HH$ is $(K,\p)$-bounded for some constant $K$ that depends only on $c$ and $k$. Conversely, if $\HH$ is $(K,\p)$-bounded, then for all sufficiently large $n$, there is an $\HH_n' \subseteq \HH_n$ with at least $(1-\eps)e(\HH_n)$ edges that satisfies~\eqref{eq:Delta-ell} for some constant $c$ that depends only on $\eps$, $k$, and $K$. One obtains such $\HH_n'$ by repeatedly deleting from $\HH_n$ edges that contain an $\ell$-set $T$ with $\deg_{\HH}(T) > c \cdot p_n^{\ell-1} e(\HH_n)/v(\HH_n)$. Finally, note that, trivially, if $\HH_n$ is $(\F,2\eps)$-dense for some family $\F \subseteq \cP(V(\HH_n))$, then every $\HH_n'$ with $e(\HH_n') \ge (1-\eps)e(\HH_n)$ is $(\F,\eps)$-dense.
\end{remark}

Our methods also yield the following `counting' analogue of Theorem~\ref{thm:Sch-weak}. This generalizes Theorem~\ref{thm:Sz}, and does not follow from the methods of~\cite{CG} or~\cite{Sch}. In the case $\alpha = 0$, it can be thought of as a strengthening of Theorem~\ref{thm:Sch-weak}, see Corollary~\ref{cor:Sz}.

\begin{thm}
  \label{thm:Sch-counting}
  Let $\HH$ be a sequence of $k$-uniform hypergraphs, let $\alpha \in [0,1)$, and let $c$ be a positive constant. Suppose that $\p \in [0,1]^\N$ is a sequence of probabilities such that for all sufficiently large $n \in \N$, and for every $\ell \in [k]$, we have 
    \[
  \Delta_\ell(\HH_n) \le c \cdot p_n^{\ell-1} \frac{e(\HH_n)}{v(\HH_n)}.
  \]
  If $\HH$ is $\alpha$-dense, then the following holds. For every positive $\delta$, there exists a constant $C$ such that for all sufficiently large $n$, if $m \ge C p_n v(\HH_n)$, then 
  \[
  |\I( \HH_n, m )| \le \binom{(\alpha + \delta) v(\HH_n)}{m}.
  \]
\end{thm}

\begin{proof}[Proof of Theorem~\ref{thm:Sch-weak}]
  Let $\alpha \in [0,1)$, let $k \in \N$, let $\p \in [0,1]^\N$, let $c \in (0,\infty)$, and let $\HH$ be a sequence of $k$-uniform hypergraphs as in the statement of Theorem~\ref{thm:Sch-weak}. Furthermore, suppose that $\HH$ is $\alpha$-dense and fix some positive $\delta$; without loss of generality, we may assume that $\delta$ is sufficiently small. Let $n \in \N$ be sufficiently large, let $\delta' = \delta/3$, and let $\F$ denote the family of all subsets of $V(\HH_n)$ with at least $(\alpha + \delta')v(\HH_n)$ elements. Since $\HH_n$ is $\alpha$-dense, it follows that $\HH_n$ is $(\F, \eps)$-dense for some small positive $\eps$ that does not depend on $n$. Let $C' = C_{\ref{thm:main}}(k,\eps,c)$. By Theorem~\ref{thm:main}, there exist a family $\S \subseteq \binom{V(\HH_n)}{\le C'p_nv(\HH_n)}$ and functions $f \colon \S \to \ol{\F}$ and $g \colon \I(\HH_n) \to \S$ such that
  \[
  g(I) \subseteq I \qquad \text{and} \qquad I \setminus g(I) \subseteq f(g(I))
  \]
  for every $I \in \I(\HH_n)$. Let $C = C'/\delta^3$ and assume that $q_n \ge Cp_n$. Let $m = (\alpha+\delta) q_n v(\HH_n)$ and, for the sake of brevity, let us write $V = V(\HH_n)$ and $q = q_n$. Observe that
  \begin{align}
    \label{eq:Pr-alpha-H-large}
    \Pr\Big(\alpha(\HH_n[V_q]) \ge m\Big) & = \Pr\Big(\text{$I \subseteq V_q$ for some $I \in \I(\HH_n,m)$}\Big) \\
    \nonumber
    & \le \sum_{S \in \S} \Pr\Big(\text{$I \subseteq V_{q}$ for some $I \in \I(\HH_n,m)$ such that $g(I) = S$}\Big).
  \end{align}
  Fix an $S \in \S$ and let $\I'_S = \{I \in \I(\HH_n,m) \colon g(I) = S\}$. We estimate the summand in the right-hand side of~\eqref{eq:Pr-alpha-H-large} as follows:
  \begin{equation}
    \label{eq:Pr-IS}
    \Pr\Big(\text{$I \subseteq V_q$ for some $I \in \I'_S$}\Big) \le \Pr\big(S \subseteq V_q\big) \cdot \Pr\left(\big|V_q \cap f(S)\big| \ge m - |S| \right).
  \end{equation}
  To see the above inequality, simply note that for every $I \in \I'_S$, we have $I \setminus S \subseteq f(S)$.

  Now, since $m = (\alpha+3\delta') q_n v(\HH_n)$ and $S \in \S$, then
  \[
  |S| \le C'p_nv(\HH_n) \le \delta^3 q v(\HH_n) \le \delta' q v(\HH_n)
  \]
  and hence $m - |S| \ge (\alpha+2\delta')qv(\HH_n)$. On the other hand, since $|f(S)| \le (\alpha + \delta')v(\HH_n)$ by the definition of $\F$, then
  \[
  \Ex\big[|V_q \cap f(S)|\big] \le (\alpha + \delta') q v(\HH_n).
  \]
  Hence, by Chernoff's inequality, we have
  \begin{equation}
    \label{eq:Pr-Vq-fS}
    \Pr\left(\big|V_q \cap f(S)\big| \ge m - |S| \right) \le \exp\left(-\frac{(\delta')^2qv(\HH_n)}{4}\right) = \exp\left(-\frac{\delta^2qv(\HH_n)}{36}\right).
  \end{equation}
  Finally, note that since $|S| \le \delta^3 q v(\HH_n)$ for every $S \in \S$, and using~\eqref{eq:binomial-1}, 
  \begin{equation}
    \label{eq:Pr-Ssum}
    \sum_{S \in \S} \Pr\big( S \subseteq V_q \big) \le \sum_{s = 0}^{\delta^3 q v(\HH_n)} \binom{v(\HH_n)}{s} q^s \le v(\HH_n) \cdot \left( \frac{e}{\delta^3} \right)^{\delta^3 q v(\HH_n)}.
  \end{equation}
  Putting~\eqref{eq:Pr-alpha-H-large}, \eqref{eq:Pr-IS}, \eqref{eq:Pr-Vq-fS}, and~\eqref{eq:Pr-Ssum} together, we obtain
  \[
  \Pr\Big(\alpha(\HH_n[V_q]) \ge m\Big) \le \sum_{S \in \S} \Pr\big(S \subseteq V_q\big) \exp\left( -\frac{\delta^2qv(\HH_n)}{36} \right) \le \exp\big( - \delta^3 q v(\HH_n) \big),
  \]
  as required.
\end{proof}

\begin{proof}[Proof of Theorem~\ref{thm:Sch-counting}]
  Let $\alpha \in [0,1)$, let $k \in \N$, let $\p \in [0,1]^\N$, let $c \in (0,\infty)$, and let $\HH$ be a sequence of $k$-uniform hypergraphs as in the statement of Theorem~\ref{thm:Sch-weak}. Furthermore, suppose that $\HH$ is $\alpha$-dense and fix some positive $\delta$. Let $n$ be sufficiently large, let $\delta' = \delta/2$, and let $\F$ denote the family of all subsets of $V(\HH_n)$ with at least $(\alpha + \delta')v(\HH_n)$ elements. Since $\HH_n$ is $\alpha$-dense, it follows that $\HH_n$ is $(\F, \eps)$-dense for some small positive $\eps$ that does not depend on $n$. Let $C' = C_{\ref{thm:main}}(k,\eps,c)$. By Theorem~\ref{thm:main}, there exist a family $\S \subseteq \binom{V(\HH_n)}{\le C'p_nv(\HH_n)}$ and functions $f \colon \S \to \ol{\F}$ and $g \colon \I(\HH_n) \to \S$ such that
  \[
  g(I) \subseteq I \qquad \text{and} \qquad I \setminus g(I) \subseteq f(g(I))
  \]
  for every $I \in \I(\HH_n)$. Let $C = C'/\delta^2$ and assume that $m \ge C p_n v(\HH_n)$. Fix an $S \in \S$, let $\I_S = \{I \in \I(\HH_n, m) \colon g(I) = S\}$, and note for future reference that
  \begin{equation}
    \label{eq:S-Sch-count}
    |S| \le C'p_nv(\HH_n) \le \delta^2 m.
  \end{equation}
  Since $f(S) \in \ol{\F}$, we have $|f(S)| < (\alpha + \delta')v(\HH_n)$. Therefore,
  \[
  |\I_S| \le \binom{|f(S)|}{m - |S|} \le \binom{(\alpha + \delta')v(\HH_n)}{m - |S|}.
  \]
  To see the above inequality, simply note that for every $I \in \I_S$, we have $I \setminus S \subseteq f(S)$. 
  
  It follows, using~\eqref{eq:binomial-3} and~\eqref{eq:binomial-2}, that
 \begin{equation}  \label{eq:IS-Sch-count}
 |\I_S| \le  \binom{(\alpha + \delta')v(\HH_n)}{m - |S|} \le \bigg( \frac{\alpha + \delta'}{\alpha + \delta} \bigg)^{m - |S|} \left( \frac{m}{(\alpha + \delta) v(\HH_n) - m} \right)^{|S|} \binom{(\alpha + \delta)v(\HH_n)}{m}.
 \end{equation}
 Now, if $m \ge (\alpha+\delta')v(\HH_n)$, then every $m$-subset of $V(\HH_n)$ belongs to $\F$ and hence there is no independent set of size $m$. We may therefore assume that $m < (\alpha + \delta')v(\HH_n) = (\alpha + \delta/2)v(\HH_n)$. Setting $s = |S|$, we obtain
  \[
  \binom{v(\HH_n)}{s} \cdot |\I_S| \le \bigg( \frac{\alpha + \delta'}{\alpha + \delta} \bigg)^{m/2} \left( \frac{e v(\HH_n)}{s} \cdot \frac{2m}{\delta v(\HH_n)} \right)^{s} \binom{(\alpha + \delta)v(\HH_n)}{m} \le e^{-\delta^2 m} \binom{(\alpha+\delta) v(\HH_n)}{m},
  \]
since $s \le \delta^2 m$, by~\eqref{eq:S-Sch-count}, and provided that $\delta$ is sufficiently small. It follows that
  \[
  |\I(\HH_n,m)| = \sum_{S \in \S} |\I_S| \le \sum_{s = 0}^{\delta^2 m} \binom{v(\HH_n)}{s} \max\big\{ |\I_S| \colon |S| = s \big\} \le \binom{(\alpha+\delta) v(\HH_n)}{m},
  \]
  as claimed.
\end{proof}

\section{Stability results for sparse sets}

\label{sec:stability-results}

In this section, we shall deduce from Theorem~\ref{thm:main} two versions of the general transference theorem for stability results proved by Conlon and Gowers~\cite{CG}. Similarly as in Section~\ref{sec:extremal-results}, we shall state our results using the terminology used by Schacht~\cite{Sch}. We remark here that in parallel to this work, Schacht's method was adapted to yield sparse random analogues of stability statements by Samotij~\cite{Sam}. The main result of this section is most easily compared with~\cite[Theorem~3.4]{Sam}. We begin by recalling the following definition from~\cite{ABMS1}.

\begin{defn}
  \label{defn:aB-stable}
  Let $\HH$ be a sequence of $k$-uniform hypergraphs, let $\alpha$ be a positive real, and let $\B$ be a sequence of sets with $\B_n \subseteq \cP(V(H_n))$. We say that $\HH$ is \emph{$(\alpha, \B)$-stable} if for every positive $\delta$, there exist positive $\eps$ and $n_0$ such that the following holds. For every $n$ with $n \ge n_0$ and every $U \subseteq V(\HH_n)$ with $|U| \ge (\alpha - \eps)v(\HH_n)$, we have either $e(\HH_n[U]) \ge \eps e(\HH_n)$ or $|U \setminus B| \le \delta v(\HH_n)$ for some $B \in \B_n$.
\end{defn}

Roughly speaking, a sequence $\HH$ of hypergraphs is $(\alpha, \B)$-stable if for every $A \subseteq V(\HH_n)$ that is almost as large as $\alpha v(\HH_n)$, the set $A$ is either very `close' to some extremal set $B \in \B_n$ or it contains `many' (a positive fraction of all) edges of $\HH_n$. Note that in many natural settings, such a property does hold, for example, as a~consequence of the~Erd{\H o}s-Simonovits stability theorem~\cite{ES1, ES2} and the~removal lemma for graphs.

We again start with the `random' version of our stability result, which was originally proved in~\cite{Sam}.

\begin{thm}
  \label{thm:stability-random}
  Let $\HH$ be a sequence of $k$-uniform hypergraphs, let $\alpha \in (0,1)$, and let $c$ be a positive constant. Let $\p$ be a sequence of probabilities such that, for every $\ell \in [k]$,
  \[
  \Delta_\ell(\HH_n) \le c \cdot p_n^{\ell-1} \frac{e(\HH_n)}{v(\HH_n)}
  \]
  and let $\B$ be a sequence of sets with $\B_n \subseteq \cP(V(\HH_n))$.
  
  If $\HH$ is $(\alpha, \B)$-stable, then the following holds. For every positive $\delta$, there exist $\eps$ and $C$ such that if $q_n \ge Cp_n$ and $q_n v(\HH_n) \to \infty$ as $n \to \infty$, then a.a.s.~every independent set $I \subseteq V(\HH_n)_{q_n}$ with $|I| \ge (\alpha - \eps) q_n v(\HH_n)$ satisfies $|I \setminus B| < \delta q_n v(\HH_n)$ for some $B \in \B_n$.
\end{thm}

The following theorem, a `counting' analogue of Theorem~\ref{thm:stability-random}, is our main stability result. A simple version of it, applicable to $3$-uniform hypergraphs with $\Delta_2(\HH_n) = O(1)$, was proved in~\cite{ABMS1} and used in~\cite{ABMS1,ABMS2} to count sum-free subsets in Abelian groups and in the set $[n]$.

\begin{thm}
  \label{thm:stability-counting}
  Let $\HH$ be a sequence of $k$-uniform hypergraphs, let $\alpha \in (0,1)$, and let $c$ be a positive constant. Let $\p$ be a sequence of probabilities such that, for every $\ell \in [k]$,
  \[
  \Delta_\ell(\HH_n) \le c \cdot p_n^{\ell-1} \frac{e(\HH_n)}{v(\HH_n)}
  \]
  and let $\B$ be a sequence of sets with $\B_n \subseteq \cP(V(\HH_n))$.
  
  If $\HH$ is $(\alpha, \B)$-stable, then the following holds. For every positive $\delta$, there exist $\eps$ and $C$ such that if $m \ge Cp_nv(\HH_n)$, then there are at most
  \[
  (1-\eps)^m \binom{\alpha v(\HH_n)}{m}
  \]
  independent sets $I \in \I(\HH_n,m)$ such that $|I \setminus B| \ge \delta m$ for every $B \in \B_n$. 
\end{thm} 

\begin{proof}[Proof of Theorem~\ref{thm:stability-random}]
  The proof is similar to the proof of Theorem~\ref{thm:Sch-weak}. Let $k \in \N$, $\alpha \in (0,1)$, $\p \in [0,1]^\N$, $c \in (0,\infty)$, and $\HH$ and $\B$ be as in the statement of Theorem~\ref{thm:stability-random}. Furthermore, suppose that $\HH$ is $(\alpha, \B)$-stable and fix some small positive $\delta$. Let $\eps$ be a small positive constant, let $n$ be sufficiently large, let $\delta' = \delta/3$ and $\eps' = 3\eps$, and set
  \[
  \F = \big\{ A \subseteq V(\HH_n) \colon |A| \ge (\alpha - \eps')v(\HH_n) \text{ and } |A \setminus B| \ge \delta' v(\HH_n) \text{ for every } B \in \B_n \big\}.
  \]
  Since $\HH$ is $(\alpha, \B)$-stable, it follows that $\HH_n$ is $(\F, \eps)$-dense, provided that $\eps$ is sufficiently small. Let $C' = C_{\ref{thm:main}}(k,\eps,c)$. By Theorem~\ref{thm:main}, there exist a family $\S \subseteq \binom{V(\HH_n)}{\le C'p_nv(\HH_n)}$ and functions $f \colon \S \to \ol{\F}$ and $g \colon \I(\HH_n) \to \S$ such that
  \[
  g(I) \subseteq I \qquad \text{and} \qquad I \setminus g(I) \subseteq f(g(I))
  \]
  for every $I \in \I(\HH_n)$. Let $C = C'/\eps^3$ and assume that $q_n \ge Cp_n$. Let $m = (\alpha-\eps) q_n v(\HH_n)$ and, for the sake of brevity, let us write $V = V(\HH_n)$ and $q = q_n$. Let
  \[
  \I' = \big\{ I \in \I(\HH_n) \colon |I| \ge m \text{ and } |I \setminus B| \ge \delta q v(\HH_n) \text{ for every } B \in \B_n \big\}
  \]
  and let $\A$ denote the event that $\HH_n[V_q]$ contains an independent set $I \in \I'$. We are required to prove that $\Pr(\A)$ tends to $0$ as $n \to \infty$.

  Observe first that
  \begin{equation}
    \label{eq:Pr-not-stable}
    \Pr(\A) \le \sum_{S \in \S} \Pr\Big(\text{$I \subseteq V_q$ for some $I \in \I'$ such that $g(I) = S$}\Big).
  \end{equation}
  Fix an $S \in \S$, let $\I'_S = \{I \in \I' \colon g(I) = S\}$, and note for future reference that 
  \begin{equation}
    \label{eq:S}
    |S| \le C'p_nv(\HH_n) \le \eps^3 q v(\HH_n).
  \end{equation}
  We claim that
  \begin{equation}
    \label{eq:Pr-not-stable-S}
    \Pr\Big(\text{$I \subseteq V_q$ for some $I \in \I'_S$}\Big) \le \Pr\big(S \subseteq V_q\big) \cdot \exp\left( -\frac{\eps^2 q v(\HH_n)}{4} \right).
  \end{equation}
  In order to prove~\eqref{eq:Pr-not-stable-S}, recall that since $f(S) \in \ol{\F}$, we either have $|f(S)| < (\alpha - \eps')v(\HH_n)$ or $|f(S) \setminus B| < \delta' v(\HH_n)$ for some $B \in \B_n$. We therefore consider two cases.

  \medskip
  \noindent
  \textbf{Case 1:} $|f(S)| < (\alpha - \eps')v(\HH_n)$.
  \medskip

  \noindent
  We bound the left-hand side of~\eqref{eq:Pr-not-stable-S} as follows:
  \begin{equation}
    \label{eq:Pr-I'S-1}
    \Pr\Big(\text{$I \subseteq V_q$ for some $I \in \I'_S$}\Big) \le \Pr\big(S \subseteq V_q\big) \cdot \Pr\left(\big|V_q \cap f(S)\big| \ge m - |S| \right).
  \end{equation}
  In order to justify the above inequality, note that for every $I \in \I'_S$, we have $I \setminus S \subseteq f(S)$. Recall that $\eps' = 3\eps$. Since $m - |S| \ge (\alpha - 2\eps)qv(\HH_n)$, by~\eqref{eq:S}, and
  \[
  \Ex[|V_q \cap f(S)|] \le (\alpha - \eps')qv(\HH_n) = (\alpha - 3\eps) q v(\HH_n),
  \]
  then by Chernoff's inequality we have
  \begin{equation}
    \label{eq:Pr-Vq-fS-stab-1}
    \Pr\left(\big|V_q \cap f(S)\big| \ge m - |S| \right) \le \exp\left(-\frac{\eps^2 q v(\HH_n)}{4}\right).
  \end{equation}
  Combining~\eqref{eq:Pr-I'S-1} and~\eqref{eq:Pr-Vq-fS-stab-1}, we obtain~\eqref{eq:Pr-not-stable-S}, as required.

  \medskip
  \noindent
  \textbf{Case 2:} $|f(S) \setminus B| < \delta' v(\HH_n)$ for some $B \in \B_n$.
  \medskip

  \noindent
  We estimate the left-hand side of~\eqref{eq:Pr-not-stable-S} as follows:
  \[
  \Pr\Big(\text{$I \subseteq V_q$ for some $I \in \I'_S$}\Big) \le \Pr\big(S \subseteq V_q\big) \cdot \Pr\Big(\big|V_q \cap (f(S) \setminus B)\big| \ge \delta q v(\HH_n) - |S| \Big).
  \]
  This follows from the definition of $\I'$ and the fact that $I \setminus S \subseteq f(S)$ for every $I \in \I'_S$. Since $|f(S) \setminus B| < \delta' v(\HH_n)$, we have
  \[
  \Ex\big[|V_q \cap (f(S) \setminus B)|\big] < \delta' q v(\HH_n),
  \]
  whereas $\delta q v(\HH_n) - |S| \ge 2\delta' q v(\HH_n)$ by~\eqref{eq:S} and since $\delta = 3\delta'$. By Chernoff's inequality, it follows that
  \[
  \Pr\left(\big|V_q \cap (f(S) \setminus B)\big| \ge 3\delta' q v(\HH_n) - |S| \right) \le \exp\left(-\frac{(\delta')^2 q v(\HH_n)}{4}\right) \le \exp\left(- \frac{\eps^2 q v(\HH_n)}{4}\right)
  \]
  since $\eps$ was chosen sufficiently small. Thus~\eqref{eq:Pr-not-stable-S} follows in this case as well.

  \medskip
  Finally, note that, since $|S| \le \eps^3 q v(\HH_n)$ for every $S \in \S$, as in~\eqref{eq:Pr-Ssum}, we have
  \begin{equation}
    \label{eq:Pr-Ssum-stable}
    \sum_{S \in \S} \Pr\big(S \subseteq V_q\big) \le \sum_{s = 0}^{\eps^3 q v(\HH_n)} \binom{v(\HH_n)}{s} q^s \le v(\HH_n) \cdot \left(\frac{e}{\eps^3} \right)^{\eps^3 q v(\HH_n)} .
  \end{equation}
  Putting~\eqref{eq:Pr-not-stable},~\eqref{eq:Pr-not-stable-S}, and~\eqref{eq:Pr-Ssum-stable} together, we obtain
  \[
  \Pr(\A) \le \sum_{S \in \S} \Pr\big(S \subseteq V_q\big) \exp\left(-\frac{\eps^2 q v(\HH_n)}{4}\right) \le \exp(-\eps^3 q v(\HH_n)),
  \]
  as required.
\end{proof}

\begin{proof}[Proof of Theorem~\ref{thm:stability-counting}]
  Let $k \in \N$, $\alpha \in (0,1)$, $\p \in [0,1]^\N$, $c \in (0,\infty)$ ,and $\HH$ and $\B$ be as in the statement of Theorem~\ref{thm:stability-counting}. Furthermore, suppose that $\HH$ is $(\alpha, \B)$-stable and fix some positive $\delta$. Let $\delta'$ be a sufficiently small positive constant (depending only on $\alpha$ and $\delta$), let $\eps$ be a small positive constant, and let $n$ be sufficiently large. Let $\eps' = 2\eps$, and set
  \[
  \F = \big\{ A \subseteq V(\HH_n) \colon |A| \ge (\alpha - \eps')v(\HH_n) \text{ and } |A \setminus B| \ge \delta' v(\HH_n) \text{ for every } B \in \B_n \big\}.
  \]
  Since $\HH$ is $(\alpha, \B)$-stable, it follows that $\HH_n$ is $(\F, \eps)$-dense, provided that $\eps$ is sufficiently small (as a function of $\delta'$). Let $C' = C_{\ref{thm:main}}(k,\eps,c)$. By Theorem~\ref{thm:main}, there exist a family $\S \subseteq \binom{V(\HH_n)}{\le C'p_nv(\HH_n)}$ and functions $f \colon \S \to \ol{\F}$ and $g \colon \I(\HH_n) \to \S$ such that
  \[
  g(I) \subseteq I \qquad \text{and} \qquad I \setminus g(I) \subseteq f(g(I))
  \]
  for every $I \in \I(\HH_n)$. Let $C = C'/\eps^2$, assume that $m \ge Cp_nv(\HH_n)$, and set
  \[
  \I' = \big\{ I \in \I(\HH_n, m) \colon |I \setminus B| \ge \delta m \text{ for every } B \in \B_n \big\}.
  \]
  Our task is to bound the size of $\I'$ from above. To this end, fix an $S \in \S$ and let $\I'_S = \{I \in \I' \colon g(I) = S\}$. Note for future reference that 
  \begin{equation}
    \label{eq:S-count}
    |S| \le C'p_nv(\HH_n) \le \eps^2 m.
  \end{equation}
  Since $f(S) \in \ol{\F}$, we either have $|f(S)| < (\alpha - \eps')v(\HH_n)$ or $|f(S) \setminus B| < \delta' v(\HH_n)$ for some $B \in \B_n$. We therefore consider two cases.

  \medskip
  \noindent
  \textbf{Case 1:} $|f(S)| < (\alpha - \eps')v(\HH_n)$.
  \medskip

  \noindent
  We claim that in this case
  \begin{equation}
    \label{eq:stab-count-case1}
    \binom{v(\HH_n)}{|S|} \cdot |\I_S'| \le \frac{(1-\eps)^m}{2m} \binom{\alpha v(\HH_n)}{m}.
  \end{equation}
  To prove~\eqref{eq:stab-count-case1}, we first estimate the size of $\I'_S$ as follows:
  \[
  |\I_S'| \le \binom{|f(S)|}{m - |S|} \le \binom{(\alpha - \eps')v(\HH_n)}{m - |S|}.
  \]
  The above inequality follows since $I \setminus S \subseteq f(S)$ for every $I \in \I'_S$. 
  
 It follows, using~\eqref{eq:binomial-3} and~\eqref{eq:binomial-2}, as in~\eqref{eq:IS-Sch-count}, that
$$
 |\I'_S| \le \binom{(\alpha - \eps')v(\HH_n)}{m - |S|} \le \bigg( \frac{\alpha-\eps'}{\alpha} \bigg)^{m - |S|} \left( \frac{m}{\alpha v(\HH_n) - m} \right)^{|S|} \binom{\alpha v(\HH_n)}{m}.
$$
 Now, if $m \ge (\alpha-\eps')v(\HH_n)$, then $\I' \subseteq \F$ and hence $\I' = \emptyset$, since $\HH_n$ is $(\F, \eps)$-dense. We may therefore assume that $m < (\alpha - \eps')v(\HH_n) = (\alpha - 2\eps)v(\HH_n)$. We obtain
   \[
  \binom{v(\HH_n)}{|S|} |\I_S'| \le \bigg( \frac{\alpha-\eps'}{\alpha} \bigg)^{m/2} \left( \frac{e v(\HH_n)}{|S|} \cdot \frac{m}{2\eps v(\HH_n)} \right)^{|S|} \binom{\alpha v(\HH_n)}{m} \le \frac{(1-\eps)^m}{2m} \binom{\alpha v(\HH_n)}{m},
  \]
since $|S| \le \eps^2 m$ and $\eps' = 2\eps$,  as claimed.
 
  \medskip
  \noindent
  \textbf{Case 2:} $|f(S) \setminus B| < \delta' v(\HH_n)$ for some $B \in \B_n$.
  \medskip

  \noindent
  We claim that in this case
  \begin{equation}
    \label{eq:stab-count-case2}
    \binom{v(\HH_n)}{|S|} \cdot |\I_S'| \le \delta^m \binom{\alpha v(\HH_n)}{m}.
  \end{equation}
  To prove~\eqref{eq:stab-count-case2}, we first estimate the size of $\I'_S$ as follows:
  \begin{equation}
    \label{eq:IS'-count-2}
    |\I_S'| \le \binom{|f(S) \setminus B|}{\delta m - |S|} \binom{|f(S)|}{m - \delta m} \le \binom{\delta' v(\HH_n)}{\delta m - |S|}\binom{v(\HH_n)}{m - \delta m}.
  \end{equation}
  To see the first inequality, recall that every $I \in \I'_S$ contains at least $\delta m - |S|$ elements of $f(S) \setminus B$ for every $B \in \B_n$. Recall that $|S| \le \eps^2m$ and note that therefore, if $m \ge (\alpha/2)v(\HH_n)$, then $\delta m - |S| \ge \delta' v(\HH_n)$ and hence $\I'_S = \emptyset$. Thus, we may assume that $m < (\alpha/2)v(\HH_n)$. It follows, using~\eqref{eq:binomial-2} and \eqref{eq:binomial-4}, that
  \begin{equation}
    \label{eq:IS'-count-2-2}
    \binom{v(\HH_n)}{m - \delta m} \le \left(\frac{m}{v(\HH_n) - m} \right)^{\delta m} \binom{v(\HH_n)}{m} \le \left(\frac{2m}{v(\HH_n)} \right)^{\delta m} \left(\frac{2}{\alpha}\right)^m \binom{ \alpha v(\HH_n)}{m}.
  \end{equation}
  Hence, by~\eqref{eq:S-count},~\eqref{eq:IS'-count-2}, and~\eqref{eq:IS'-count-2-2}, using~\eqref{eq:binomial-1}, we have
  \begin{align*}
    \binom{v(\HH_n)}{|S|} |\I_S'| & \le \left( \frac{ev(\HH_n)}{|S|} \right)^{|S|} \left( \frac{2e\delta' v(\HH_n)}{\delta m}\right)^{\delta m - |S|} \left(\frac{2m}{v(\HH_n)} \right)^{\delta m} \left(\frac{2}{\alpha}\right)^m \binom{ \alpha v(\HH_n)}{m}\\
    & \le \left( \frac{1}{|S|} \cdot \frac{\delta m}{2\delta'} \right)^{|S|} \left( \frac{4e\delta'}{\delta} \right)^{\delta m} \left(\frac{2}{\alpha}\right)^m \binom{ \alpha v(\HH_n)}{m} \le \delta^m \binom{\alpha v(\HH_n)}{m},
    \end{align*}
    as claimed, since $|S| \le\eps^2 m$ and $\delta'$ and $\eps$ were chosen to be sufficiently small. Indeed, note that (for this calculation, and assuming that $\delta$ is sufficiently small) $\delta' = \delta^3 \cdot (\delta \alpha / 2e)^{1/\delta}$ and $\eps < \delta'$ suffice.

  \medskip

  Finally, by~\eqref{eq:stab-count-case1} and~\eqref{eq:stab-count-case2}, we obtain
  \[
  |\I'| = \sum_{S \in \S} |\I_S'| \le \sum_{s = 0}^{\eps^2 m} \binom{v(\HH_n)}{s} \max\big\{ |\I_S'| \colon |S| = s \big\} \le (1-\eps)^m \binom{\alpha v(\HH_n)}{m},
  \]
  as claimed.
\end{proof}

\section{Tur{\'a}n's problem in random graphs}

\label{sec:Turan}

In this section, we shall deduce from Theorems~\ref{thm:Sch-weak} and~\ref{thm:stability-random} the sparse random analogues of the classical theorems of Erd{\H o}s and Stone~\cite{ErSt} and Tur{\'a}n~\cite{Turan} and of Erd{\H o}s and Simonovits~\cite{ES1,ES2}, Theorems~\ref{thm:Turan-Gnp} and~\ref{thm:stability-Gnp}. In fact, we will prove a natural generalization of Theorem~\ref{thm:Turan-Gnp} to $t$-uniform hypergraphs, Theorem~\ref{thm:t-Turan-Gnp} below, which was already proved by Conlon and Gowers~\cite{CG} and Schacht~\cite{Sch}. We first recall the following generalization of the notion of $2$-density of a graph to $t$-uniform hypergraphs.

\begin{defn}
  Let $H$ be a $t$-uniform hypergraph with at least $t + 1$ vertices. We define the \emph{$t$-density} of $H$, denoted by $m_t(H)$, by
  \[
  m_t(H) = \max \left\{ \frac{e(H') - 1}{v(H') - t} \colon H' \subseteq H \text{ with } v(H') \ge t + 1 \right\}.
  \]
\end{defn}

We also recall that the \emph{Tur{\'a}n density} of a $t$-uniform hypergraph $H$, denoted $\pi(H)$, is defined by
\begin{equation}
  \label{eq:piH}
  \pi(H) = \lim_{n \to \infty} \frac{\ex\big(K_n^{t}, H\big)}{\binom{n}{t}},
\end{equation}
where, as usual, $\ex\big(K_n^{t}, H \big)$ is the Tur{\'a}n number for $H$, that is, the maximum number of edges in an $H$-free $t$-uniform hypergraph with $n$ vertices.

\begin{thm}
  \label{thm:t-Turan-Gnp}
  For every $t$-uniform hypergraph $H$ with $\Delta(H) \ge 2$ and every positive $\delta$, there exists a positive constant $C$ such that if $q_n \ge Cn^{-1/m_t(H)}$, then
  \[
  \Pr\left( \ex\big(G^t(n,q_n), H\big) \le (\pi(H) + \delta) q_n \binom{n}{t} \right) \to 1
  \]
  as $n \to \infty$.
\end{thm}

Once again, we emphasize that we actually obtain essentially optimal bounds on the probability in the above statement, i.e., bounds of the form  $1 - \exp(-bq_nn^t)$ for some positive constant $b$ that depends only on $H$ and $\delta$.

Theorems~\ref{thm:t-Turan-Gnp} and~\ref{thm:stability-Gnp}, and hence also Theorem~\ref{thm:Turan-Gnp}, will follow easily from our general transference results, Theorems~\ref{thm:Sch-weak} and~\ref{thm:stability-random}, the classical supersaturation results of Erd{\H o}s and Simonovits~\cite{ES83} (for Theorem~\ref{thm:t-Turan-Gnp}), and the stability theorem of Erd{\H o}s and Simonovits~\cite{ES1,ES2} together with the so-called graph removal lemma (for Theorem~\ref{thm:stability-Gnp}). We only need to check that the hypergraph of copies of $H$ in the complete hypergraph $K_n^t$, to which we would like to apply our transference theorems, satisfies the assumptions of Theorems~\ref{thm:Sch-weak} and \ref{thm:stability-random}. Since we are going to use this fact several times in this and later sections, we state it as a separate proposition. 

Let $H$ be an arbitrary $t$-uniform hypergraph. The \emph{hypergraph of copies of $H$ in $K_n^t$} is the $e(H)$-uniform hypergraph on the vertex set $E(K_n^t)$ whose edges are the edge sets of all copies of $H$ in $K_n^t$.

\begin{prop}
  \label{prop:Delta-bal-hyp}
  Let $n$ and $t$ be integers with $t \ge 2$ and let $H$ be a $t$-uniform hypergraph. Set $k = e(H)$ and let $\HH$ be the $k$-uniform hypergraph of copies of $H$ in $K_n^t$. There exists a positive constant $c$ such that, letting $p = n^{-1/m_t(H)}$, 
\begin{equation} \label{item:Delta-bal-hyp-b}
\Delta_\ell(\HH) \le c \cdot p^{\ell-1} \frac{e(\HH)}{v(\HH)} 
\end{equation}
for every $\ell \in [k]$.
\end{prop}

\begin{proof}
  Note that $v(\HH) = \binom{n}{t} = \Theta(n^t)$ and that $e(\HH) = \frac{(v(H))!}{|\Aut(H)|} \cdot \binom{n}{v(H)} = \Theta\big( n^{v(H)} \big)$. By the definition of $p$ and $m_t(H)$, we have 
  \begin{equation}\label{eq:pbound:def:mtH}
    p^{e(H')-1}n^{v(H')-t} \ge 1
  \end{equation}
  for every $H' \subseteq H$. Now, for each $\ell \in [k]$,
  \[
  \Delta_\ell(\HH) \le c' \cdot \max\left\{ n^{v(H) - v(H')} \colon H' \subseteq H \text{ with } e(H') = \ell \right\}
  \]
  for some positive constant $c'$. Since $e(\HH) / v(\HH) \ge c'' \cdot n^{v(H)-t}$ for some constant $c''$, it follows that
  \begin{align*}
    \Delta_\ell(\HH) \cdot \left( p^{\ell-1} \frac{e(\HH)}{v(\HH)} \right)^{-1} & \le c' \cdot \frac{v(\HH)}{e(\HH)} \cdot \max_{H' \subseteq H \colon e(H') = \ell} \left( \frac{n^{v(H)}}{p^{e(H')-1}n^{v(H')}} \right) \\
    & \le \frac{c'}{c''} \cdot  \max_{H' \subseteq H \colon e(H') = \ell} \left( \frac{1}{p^{e(H')-1}}{n^{v(H')-t}} \right) \le \frac{c'}{c''},
  \end{align*}
  where the last inequality follows by~\eqref{eq:pbound:def:mtH}.
\end{proof}

\begin{proof}[Proof of Theorem~\ref{thm:t-Turan-Gnp}]
  Let $H$ be a $t$-uniform hypergraph, let $k = e(H)$, and let $(\HH_n)_{n \in \N}$ be the sequence of $k$-uniform hypergraphs of copies of $H$ in $K_n^t$. Let $\alpha = \pi(H)$, let $\delta$ be a positive constant, and let $p_n = n^{-1/m_t(H)}$. It follows easily from the supersaturation theorem of Erd{\H o}s and Simonovits~\cite{ES83} that $\HH$ is $\alpha$-dense, see~\cite{Sch}. Let $C = C_{\ref{thm:Sch-weak}}(\HH, \delta)$ and assume that $q_n \ge Cp_n = Cn^{-1/m_t(H)}$. Note that the assumption that $H$ contains a vertex of degree at least $2$ implies that $m_t(H) > 1/t$ and hence $q_n v(\HH_n) \to \infty$ as $n \to \infty$. Together with Proposition~\ref{prop:Delta-bal-hyp}, this implies that $\HH$ satisfies the assumptions of Theorem~\ref{thm:Sch-weak} and hence with probability tending to $1$ as $n \to \infty$,
\[
\ex\big( G^t(n,q_n), H \big) = \alpha\left( \HH_n\big[E(G^t(n,q_n))\big] \right) \le ( \pi(H) + \delta ) q_n \binom{n}{t},
\]
as required.
\end{proof}

In the proof of Theorems~\ref{thm:stability-Gnp} and~\ref{thm:H-free-structure}, we shall need the following proposition, which is a fairly straightforward consequence of the Erd{\H o}s-Simonovits stability theorem~\cite{ES1,ES2} and the graph removal lemma~\cite{ErFrRo}. A proof of this statement can be found in~\cite{Sam}. We remark that a new proof of the graph removal lemma, which avoids the use of the Szemer\'edi regularity lemma, was given recently by Fox~\cite{Fox}. 

\begin{prop}
  \label{prop:stability-removal}
  For every graph $H$ and every positive $\delta$, there exists a positive $\eps$ such that the following holds for every $n \in \N$. If $G$ is an $n$-vertex graph with 
  \[
  e(G) \ge \left( 1 - \frac{1}{\chi(H) - 1} - \eps \right) \binom{n}{2},
  \] 
  then either $G$ may be made $(\chi(H) - 1)$-partite by removing from it at most $\delta n^2$ edges or $G$ contains at least $\eps n^{v(H)}$ copies of $H$.
\end{prop}

\begin{proof}[Proof of Theorem~\ref{thm:stability-Gnp}]
  Let $H$ be a graph, let $k = e(H)$, and let $(\HH_n)_{n \in \N}$ be the sequence of $k$-uniform hypergraphs of copies of $H$ in $K_n$. Let $\alpha = \pi(H) = \left(1 - \frac{1}{\chi(H)-1}\right)$, let $\delta$ be a positive constant, and let $p_n = n^{-1/m_2(H)}$. Moreover, let $\B_n$ be the family of all complete $(\chi(H)-1)$-partite subgraphs of $K_n$. By Proposition~\ref{prop:stability-removal}, $\HH$ is $(\alpha,\B)$-stable. Let $C = C_{\ref{thm:stability-random}}(\HH, \delta)$, let $\eps = \eps_{\ref{thm:stability-random}}(\HH, \delta)$, and assume that $q_n \ge Cp_n = Cn^{-1/m_2(H)}$. Note that the assumption that $H$ contains a vertex of degree at least $2$ implies that $m_2(H) > 1/2$ and hence $q_n v(\HH_n) \to \infty$ as $n \to \infty$. Together with Proposition~\ref{prop:Delta-bal-hyp}, the discussion above implies that $\HH$ satisfies the assumptions of Theorem~\ref{thm:stability-random} and hence with probability tending to $1$ as $n \to \infty$, every independent set $G' \subseteq G(n,q_n)$ with $|G'| \ge (\alpha-\eps) q_n v(\HH_n)$ satisfies $|G' \setminus B| \le \delta q_n v(\HH_n)$ for some $B \in \B_n$. In other words, with probability tending to $1$ as $n \to \infty$, every $H$-free subgraph of $G(n,q_n)$ with at least $\left(1-\frac{1}{\chi(H)-1}-\eps\right)\binom{n}{2}q_n$ edges can be made $(\chi(H)-1)$-partite by removing from it at most $\delta q_n \binom{n}{2}$ edges, as required.
\end{proof}

\section{The typical structure of $H$-free graphs}

\label{sec:Turan-counting}

In this section, we shall deduce from Theorems~\ref{thm:Sch-counting} and~\ref{thm:stability-counting} the sparse analogue of the theorem of Erd{\H o}s, Frankl, and R{\"o}dl~\cite{ErFrRo}, Theorem~\ref{thm:ErFrRo-sparse}, and an approximate sparse analogue of the result of Erd{\H o}s, Kleitman, and Rothschild~\cite{ErKlRo}, Theorem~\ref{thm:H-free-structure}. We stress once again that neither proof employs Szemer{\'e}di's regularity lemma. In order to prove Theorem~\ref{thm:ErFrRo-sparse}, we are actually going to prove the following natural generalization of it to $t$-uniform hypergraphs. Generalizing the definition stated in Section~\ref{sec:typical-structure}, given integers $n$ and $m$ with $0 \le m \le \binom{n}{t}$ and a $t$-uniform hypergraph $H$, let us denote by $f_{n,m}(H)$ the number of $H$-free $t$-uniform hypergraphs on the vertex set $[n]$ that have exactly $m$ edges. 

\begin{thm}
  \label{thm:t-ErFrRo-sparse}
  For every $t$-uniform hypergraph $H$ and every positive $\delta$, there exists a positive constant $C$ such that the following holds. For every $n \in \N$, if $m \ge Cn^{t-1/m_t(H)}$, then
  \[
  \binom{\ex(n,H)}{m} \le f_{n,m}(H) \le \binom{\ex(n,H) + \delta n^t}{m}.
  \]
\end{thm}

We remark that Theorem~\ref{thm:t-ErFrRo-sparse} refines a result of Nagle, R\"odl, and Schacht~\cite{NRS}, who, using the hypergraph regularity lemma, generalized~\eqref{eq:fnH-upper} to $t$-uniform hypergraphs.

\begin{proof}[Proof of Theorem~\ref{thm:t-ErFrRo-sparse}]
  Let $H$ be a $t$-uniform hypergraph, let $k = e(H)$, and let $(\HH_n)_{n \in \N}$ be the sequence of $k$-uniform hypergraphs of copies of $H$ in $K_n^t$. Let $\alpha = \pi(H)$, see~\eqref{eq:piH}, let $\delta$ be a positive constant, and let $p_n = n^{-1/m_t(H)}$. It follows easily from the supersaturation theorem of Erd{\H o}s and Simonovits~\cite{ES83} that $\HH$ is $\alpha$-dense, see~\cite{Sch}. Let $C = C_{\ref{thm:Sch-counting}}(\HH, \delta)$ and assume that $m \ge C n^{t-1/m_t(H)} \ge C p_n v(\HH_n)$. Note that Proposition~\ref{prop:Delta-bal-hyp} implies that $\HH$ satisfies the assumptions of Theorem~\ref{thm:Sch-counting} and hence
\[
f_{n,m}(H) = |I(\HH_n, m)| \le \binom{(\pi(H) + \delta)\binom{n}{t}}{m} \le \binom{\ex(n,H) + \delta n^t}{m},
\]
as required. The claimed lower bound on $f_{n,m}(H)$ is trivial.
\end{proof}

\begin{proof}[Proof of Theorem~\ref{thm:H-free-structure}]
   Let $H$ be a graph, let $k = e(H)$, and let $(\HH_n)_{n \in \N}$ be the sequence of $k$-uniform hypergraphs of copies of $H$ in $K_n$. Let $\alpha = \pi(H) = 1 - \frac{1}{\chi(H)-1}$, let $\delta$ be a positive constant, and let $p_n = n^{-1/m_2(H)}$. Moreover, let $\B_n$ be the family of all complete $(\chi(H)-1)$-partite subgraphs of $K_n$. By Proposition~\ref{prop:stability-removal}, $\HH$ is $(\alpha,\B)$-stable. Let $C = C_{\ref{thm:stability-counting}}(\HH, \delta)$, let $\eps = \eps_{\ref{thm:stability-counting}}(\HH, \delta)$, and assume that $m \ge Cn^{2-1/m_2(H)} \ge Cp_n v(\HH_n)$. Together with Proposition~\ref{prop:Delta-bal-hyp}, this implies that $\HH$ satisfies the assumptions of Theorem~\ref{thm:stability-counting} and hence, letting $f_{n,m}^\delta(H)$ denote the number of $H$-free graphs on the vertex set $[n]$ that have exactly $m$ edges and that are not $(\delta, \chi(H)-1)$-partite,
\[
f_{n,m}^\delta(H) \le (1-\eps)^m \binom{\pi(H)\binom{n}{2}}{m}.
\]
Finally, note that (trivially),
\[
f_{n,m}(H) \ge \binom{\pi(H)\binom{n}{2}}{m}
\]
and hence $f_{n,m}^\delta(H) = o\big(f_{n,m}(H)\big)$, as claimed.
\end{proof}

For $t$-uniform hypergraphs, there is no general stability theorem known; however, such results have been proved for a few specific hypergraphs (see~\cite{FPS05, FuSi05, KeMu04, KeSu05}), and in each case we obtain a corresponding result for sparse hypergraphs. For example, following~\cite{BaMu12}, let $F_5$ denote the `3-uniform triangle', i.e., the hypergraph with edge set isomorphic to $\{123,124,345\}$, and say that a 3-uniform hypergraph is \emph{triangle-free} if it contains no copy of $F_5$. The following theorem follows easily, as above, from Theorem~\ref{thm:stability-counting} combined with the hypergraph removal lemma of Gowers~\cite{Gow07} and R\"odl and Skokan~\cite{RoSk06} and the stability theorem for 3-uniform triangle-free hypergraphs, which was proved by Keevash and Mubayi~\cite{KeMu04}. 

\begin{thm}\label{3trianglestab}
For every positive $\delta$, there exists a constant $C$ such that the following holds. If $m \ge Cn^2$, then almost every triangle-free $3$-uniform hypergraph with $n$ vertices and $m$ edges can be made tripartite by removing from it at most $\delta m$ edges.
\end{thm}

\begin{proof}
Let $(\HH_n)_{n \in \N}$ be the sequence of $3$-uniform hypergraphs of copies of $F_5$ in $K_n^3$, set $\alpha = 2/9$,  and let $\B_n$ denote the collection of all complete tripartite subhypergraphs of $K_n^3$. By the hypergraph removal lemma~\cite[Theorem~1.3]{RoSk06}, combined with the stability theorem for triangle-free 3-uniform hypergraphs~\cite[Theorem~1.6]{KeMu04}, it follows that $\HH$ is $(\alpha,\B)$-dense. 

It follows by Proposition~\ref{prop:Delta-bal-hyp} that $\HH$ satisfies the conditions of Theorem~\ref{thm:stability-counting} with $p_n = n^{-1}$. Hence the number of triangle-free $3$-uniform hypergraphs with $n$ vertices and $m$ edges that cannot be made tripartite by removing at most $\delta m$ edges is at most
$$(1-\eps)^m \binom{\pi(F_5) {n \choose 3}}{m},$$
which easily implies the theorem.
\end{proof}

Finally, we remark that Theorem~\ref{3trianglestab} can be seen as an approximate sparse analogue of a result of Balogh and Mubayi~\cite{BaMu12}, who used the hypergraph regularity lemma and~\cite[Theorem~1.6]{KeMu04} to show that almost all triangle-free 3-uniform hypergraphs are tripartite. For similar results for other forbidden hypergraphs, see~\cite{BaMu11} and~\cite{PS09}.

\section{The K\L R Conjecture}

\label{sec:KLR}

In this section, we shall deduce from Theorem~\ref{thm:main} the K\L R conjecture, Theorem~\ref{thm:KLR}. As in the preceding sections, the proof will be a fairly straightforward application of Theorem~\ref{thm:main} to an appropriately defined hypergraph $\HH$ and family $\F \subseteq \cP(V(\HH))$. Let $H$ be an arbitrary graph and let $\HH$ be the $e(H)$-uniform hypergraph of canonical copies of $H$ in the complete blow-up of $H$. Defining an appropriate family $\F$ and showing that $\HH$ is $(\F,\eps)$-dense will require some work.

Given a graph $H$ and integers $n_1, \ldots, n_{v(H)}$, let us denote by $\G(H;n_1,\ldots,n_{v(H)})$ the collection of all graphs $G$ constructed in the following way. The vertex set of $G$ is a disjoint union $V_1 \cup \ldots \cup V_{v(H)}$ of sets of sizes $n_1, \ldots, n_{v(H)}$, respectively, one for each vertex of $H$. The only edges of $G$ lie between those pairs of sets $(V_i, V_j)$ such that $\{i,j\}$ is an edge of $H$. Recall the definition of $\G(H,n,m,p,\eps)$ from Section~\ref{sec:KLR-intro} and observe that $\G(H,n,m,p,\eps) \subseteq \G(H;n,\ldots,n)$ for all $m$, $p$, and $\eps$. 

The following lemma, which is a robust version of the embedding lemma, stated in Section~\ref{sec:KLR-intro}, suggests the right choice of $\F$. The lemma is well-known, and so we omit the (standard) proof.

\begin{lemma}
  \label{lemma:hole}
  Let $H$ be a graph and let $\delta \colon (0,1] \to (0,1)$ be an arbitrary function. There exist positive constants $\alpha_0$, $\xi$, and $N$ such that for every collection of integers $n_1, \ldots, n_{v(H)}$ satisfying $n_1, \ldots, n_{v(H)} \ge N$ and every graph $G \in \G(H;n_1, \ldots, n_{v(H)})$, one of the following holds:
  \begin{enumerate}[(a)]
  \item
    \label{item:hole-a}
    $G$ contains at least $\xi n_1 \dots n_{v(H)}$ canonical copies of $H$.
  \item
    \label{item:hole-b}
    There exist a positive constant $\alpha$ with $\alpha \ge \alpha_0$, an edge $\{i,j\} \in E(H)$, and sets $A_i \subseteq V_i$, $A_j \subseteq V_j$ such that $|A_i| \ge \alpha n_i$, $|A_j| \ge \alpha n_j$, and 
    $d_G(A_i,A_j) < \delta(\alpha)$.
  \end{enumerate}
\end{lemma}

Our next lemma is also straightforward. It allows us to count $(\eps, p)$-regular subgraphs of a graph that has a `hole', as in Lemma~\ref{lemma:hole}(\ref{item:hole-b}). Recall that $\G(K_2,n,m,p,\eps)$ denotes the collection of all $(\eps, p)$-regular bipartite graphs with $m$ edges and $n$ vertices in each part. Given such $G$, let $V_1(G)$ and $V_2(G)$ denote the two parts. For each $\beta \in (0,1)$, define a function $\delta \colon (0,1] \to (0,1)$ by setting
\begin{equation}
  \label{eq:delta-def}
  \delta(x) = \frac{1}{4e} \left( \frac{\beta}{2} \right)^{2/x^2}
\end{equation}
for each $x \in (0,1]$. The following lemma says that a graph $\Gt$ that has a hole of size $\alpha n$ and density at most $\delta(\alpha)$ has very few subgraphs in $\G(K_2,n,m,m/n^2,\eps)$.

\begin{lemma}\label{lemma:holecount}
  For every positive $\alpha_0$ and $\beta$, there exists a positive constant $\eps$ such that the following holds. Let $\Gt \subseteq K_{n,n}$ be such that there exist subsets $A \subseteq V_1(\Gt)$ and $B \subseteq V_2(\Gt)$ with
  \[
  \min\{ |A|,|B|\} \ge \alpha n \quad \text{and} \quad d_G(A,B) < \delta(\alpha)
  \]
  for some $\alpha \in [\alpha_0,1]$, and let $S \subseteq \Gt$. Then, for every $m$ with $|S|/\eps \le m \le n^2$, there are at most 
  \[
  \beta^{m} \binom{n^2}{m - |S|}
  \]
  subgraphs of $\Gt$ that belong to $\G(K_2,n,m,m/n^2,\eps)$ and contain $S$.
\end{lemma}

\begin{proof}
  We begin by noting that, by choosing random subsets of $A$ and $B$ if necessary, we may assume that $|A| = |B| = \alpha n$. Set $\eps = \min\{\alpha_0^2/4, 1/4\}$, write $\G^*$ for the family of all subgraphs of $\Gt$ that belong to $\G(K_2,n,m,m/n^2,\eps)$ and contain $S$, and let $G \in \G^*$. In particular, $G$ is $(\eps,p)$-regular, where $p = m/n^2$ and since $\eps \le \alpha$, it follows that the pair $(A,B)$ must have density at least $(1-\eps)p$ in $G$, and hence must contain at least $(1 - \eps - \eps/\alpha^2)p|A||B|$ edges of $E(G) \setminus S$, since $|S| \le \eps m$. Set $m' = m - |S|$ and $\eps' = \eps( 1 + 1/\alpha^2) \le 1/2$, and let us write $e_{\Gt}(A,B)$ for the number of edges of $\Gt$ that lie between the sets $A$ and $B$. Since $d_{\Gt}(A,B) < \delta(\alpha)$, then the number of choices for $G$ can be estimated as follows:
  \begin{equation}
    \label{eq:holecount-1}
    |\G^*| \le \sum_{\ell \ge (1- \eps')p|A||B|} \binom{e_{\Gt}(A,B)}{\ell} \binom{e(\Gt) - e_{\Gt}(A,B)}{m' - \ell} \le \sum_{\ell \ge \alpha^2 m/2} \binom{\delta(\alpha) \alpha^2 n^2}{\ell} \binom{n^2}{m' - \ell}.
  \end{equation}
  Note that the right-hand side of~\eqref{eq:holecount-1} is zero if $m > 2\delta(\alpha) n^2$, so we may assume that $m' \le m \le 2\delta(\alpha) n^2 \le n^2/2$. Thus, using~\eqref{eq:binomial-1} and~\eqref{eq:binomial-2}, \eqref{eq:holecount-1} implies that
  \begin{equation}
    \label{eq:holecount-2}
    |\G^*| \le \sum_{\ell \ge \alpha^2 m/2} \left( \frac{e \delta(\alpha) \alpha^2 n^2}{\ell} \right)^\ell \left( \frac{m'}{n^2 - m'} \right)^\ell \binom{n^2}{m'} \le \sum_{\ell \ge \alpha^2 m/2} \left( \frac{2e \delta(\alpha) \alpha^2 m}{\ell} \right)^\ell \binom{n^2}{m'}.
  \end{equation}
  Since $\delta(\alpha) < 1/4e$, the summand in the right-hand side of~\eqref{eq:holecount-2} is decreasing in $\ell$ on $(\alpha^2 m/2,\infty)$ and hence
  \[
  |\G^*| \le m \big( 4e \delta(\alpha) \big)^{\alpha^2 m/2} \binom{n^2}{m'} \le \beta^m \binom{n^2}{m'},
  \]
  as required, since $\big( 4e \delta(\alpha) \big)^{\alpha^2/2} = \beta/2$. 
\end{proof}

We can now easily deduce Theorem~\ref{thm:KLR} from Theorem~\ref{thm:main}.

\begin{proof}[Proof of Theorem~\ref{thm:KLR}]
  Let $H$ be a fixed graph, let $n \in \N$, and let $H(n)$ be the largest graph in the family $\G(H;n, \ldots, n)$, i.e., the complete blow-up of $H$, where each vertex of $H$ is replaced by an independent set of size $n$ and each edge of $H$ is replaced by the complete bipartite graph $K_{n,n}$. Let $\HH$ be the $e(H)$-uniform hypergraph on the vertex set $E(H(n))$ whose edges are all $n^{v(H)}$ canonical copies of $H$ in $H(n)$. 
  
  Fix an arbitrary positive constant $\beta$, let $\delta \colon (0,1] \to (0,1)$ be the function defined in~\eqref{eq:delta-def} with $\beta$ replaced by $\beta/2$, i.e., set
  \[
  \delta(x)  = \frac{1}{4e} \left( \frac{\beta}{4} \right)^{2/x^2}
  \]   
  for each $x \in (0,1]$, and let $\alpha_0 = (\alpha_0)_{\ref{lemma:hole}}(H, \delta)$, $\xi = \xi_{\ref{lemma:hole}}(H,\delta)$, and $N = N_{\ref{lemma:hole}}(H,\delta)$. Let $\F$ be the family of all subgraphs of $H(n)$, i.e., graphs in $\G(H;n, \ldots, n)$, for which (\ref{item:hole-b}) in Lemma~\ref{lemma:hole} is \emph{not} satisfied. Clearly $\F$ is an upset, and so, by Lemma~\ref{lemma:hole}, $\HH$ is $(\F,\xi)$-dense provided that $n \ge N$. 
  
  Now, since $\HH$ is contained in the hypergraph of all copies of $H$ in the complete graph on $v(H)n$ vertices and contains a positive proportion of those copies, it follows from Proposition~\ref{prop:Delta-bal-hyp} that $\HH$ satisfies the assumptions of Theorem~\ref{thm:main} with $p = n^{2 - 1/m_2(H)}$ and $\eps = \xi$, for some constant $c$ depending only on $H$. Therefore, there is a constant $C'$, a family $\S \subseteq \binom{E(H(n))}{\le C'n^{2-1/m_2(H)}}$, and functions $f \colon \S \to \ol{\F}$ and $g \colon \I(\HH) \to \S$ such that 
  \[
  g(I) \subseteq I \qquad \text{and} \qquad I \setminus g(I) \subseteq f(g(I))
  \]
  for every $I \in \I(\HH)$.

  Let $\eps$ be a sufficiently small positive constant such that, in particular, $\eps \le \eps_{\ref{lemma:holecount}}(\alpha_0, \beta/2)$, let $C = C' / \eps$, and suppose that $m \ge Cn^{2-1/m_2(H)}$. Let $\G^* = \G^*(H,n,m,m/n^2,\eps)$ and note that $\G^* \subseteq \I(\HH)$. We are required to bound from above the number of graphs in $\G^*$. 

  To this end, fix an $S \in \S$, let
  \[
  \G_S^* = \big\{ G \in \G^* \colon g(G) = S \big\},
  \]
  and let $G_S = f(S)$. For each $\{i,j\} \in E(H)$, let $s(i,j) =  e_S(V_i, V_j)$ and note that $\sum_{ij \in E(H)} s(i,j) = |S|$. Since
  \[
  |S| \le C'n^{2-1/m_2(H)} \le \eps \cdot Cn^{2-1/m_2(H)} \le \eps m,
  \]
  then $s(i,j) \le \eps m$ for every $\{i,j\} \in E(H)$. 
  
  Now, since $G_S \in \ol{\F}$, it follows that there exist an $\alpha \in [\alpha_0,1]$, an edge $\{i, j\} \in E(H)$, and sets $A_i \subseteq V_i$, $A_j \subseteq V_j$ such that $|A_i|, |A_j| \ge \alpha n$ and $d_{G_S}(A_i,A_j) < \delta(\alpha)$. By Lemma~\ref{lemma:holecount}, it follows that there are at most 
  \[
  \left( \frac{\beta}{2} \right)^{m} \binom{n^2}{m - s(i,j)}
  \]
  choices for the edges between $V_i$ and $V_j$ such that $G[V_i,V_j] \in \G(K_2,n,m,m/n^2,\eps)$ and $S[V_i,V_j] \subseteq G[V_i,V_j] \subseteq S \cup G_S[V_i,V_j]$. It follows immediately that 
  \[
  |\G_S^*| \le \left(\frac{\beta}{2}\right)^m \prod_{ij \in E(H)} {n^2 \choose m - s(i,j)}.
  \]
  Summing over sets $S \in \S$, and using~\eqref{eq:binomial-1} and \eqref{eq:binomial-2}, we obtain
  \begin{align*}
    |\G^*| & \le \sum_{S \in \S} \left(\frac{\beta}{2}\right)^m \prod_{ij \in E(H)} \left( \frac{m}{n^2 - m} \right)^{s(i,j)} \binom{n^2}{m} = \left(\frac{\beta}{2}\right)^m \binom{n^2}{m}^{e(H)} \sum_{S \in \S} \left( \frac{m}{n^2 - m} \right)^{|S|} \\
    & \le \left(\frac{\beta}{2}\right)^m \binom{n^2}{m}^{e(H)} \sum_{s \le \eps m} \binom{e(H)n^2}{s} \left( \frac{2m}{n^2} \right)^{s} \le \left(\frac{\beta}{2}\right)^m \binom{n^2}{m}^{e(H)} \sum_{s \le \eps m} \left( \frac{2e \cdot e(H)m}{s} \right)^s.
  \end{align*}
  Now, since $\eps$ was chosen to be sufficiently small, it follows that the summand above is increasing in $s$ on $(0,\eps m]$ and hence 
  \[
  |\G^*| \le \left(\frac{\beta}{2}\right)^m \binom{n^2}{m}^{e(H)} m \left( \frac{2e \cdot e(H)}{\eps} \right)^{\eps m} \le \beta^m  {n^2 \choose m}^{e(H)},
  \] 
  as required.
\end{proof}

\noindent
{\bf Acknowledgement.}
The third author would like to thank Noga Alon and David Conlon for stimulating discussions. The authors would also like to thank David Conlon and Yoshiharu Kohayakawa for helpful comments on the manuscript, and David Saxton for pointing out the usefulness of allowing multiple edges. Finally, we would like to thank the anonymous referee for a very careful reading of the proof, and a plenitude of helpful suggestions.

\bibliographystyle{amsplain}
\bibliography{BMS_k-uniform}

\end{document}